\documentclass[a4paper,11pt]{elsarticle}
\usepackage{amsmath,amssymb,amsfonts}
\usepackage{amsthm}
\usepackage{bbm}
\usepackage{graphicx,color}
\usepackage{mathtools}
\usepackage{enumerate}
\usepackage{appendix}
\usepackage[colorlinks=true]{hyperref}
\hypersetup{urlcolor=blue, citecolor=red, linkcolor=blue}

\newtheorem{lemma}{Lemma}
\newtheorem{prop}{Proposition}
\newtheorem{theorem}{Theorem}
\newtheorem{corollary}{Corollary}

\newtheorem{remark}[theorem]{Remark}
\newtheorem{definition}{Definition}
\numberwithin{equation}{section}
\journal{}
\begin{document}
\nocite{*}
\begin{frontmatter}
\title{Solution space characterisation of perturbed linear discrete and continuous stochastic Volterra convolution equations: the $\ell^p$ and $L^p$ cases}

\cortext[cor1]{Corresponding author}

\author[1]{John A. D. Appleby\fnref{fn1}}
\ead{john.appleby@dcu.ie}

\author[2]{Emmet Lawless\corref{cor1}\fnref{fn2}}
\ead{elawless@umich.edu}

\affiliation[1]{organization={School of Mathematical Sciences, Dublin City University},
city={Dublin},
country={Ireland}
}
\affiliation[2]{organization={Department of Mathematics, University of Michigan},
city={Ann Arbor},
country={United States}
}

\journal{ArXiv}

\fntext[fn1]{JA is supported by the RSE Saltire Facilitation Network on Stochastic Differential Equations: Theory, Numerics and Applications (RSE1832).}

\fntext[fn2]{This research was carried out while EL was affiliated with Dublin City University and was supported by Science Foundation Ireland (16/IA/4443). This author's current affiliation is the Department of Mathematics, University of Michigan, Ann Arbor, United States.}

\begin{abstract}
In this article, we are concerned with characterising when solutions of perturbed linear stochastic Volterra summation equations are almost surely $p$-summable and when their continuous time counterparts, perturbed linear stochastic Volterra integro-differential equations, are almost surely $p$-integrable. In the discrete case, we find it necessary and sufficient that perturbing functions are $p$-summable in order to ensure paths of the discrete equation are almost surely $p$-summable, while in the continuous case, it transpires one can have almost surely $p$-integrable sample paths with non-integrable perturbation functions. For the continuous equation, the main converse is clinched by considering an appropriate discretisation and applying results from the discrete case. We also conduct a thorough study of the asymptotic behaviour of the trajectories of solutions to the continuous equation in the regime of $p$-integrable paths and provide a characterisation of almost sure convergence to zero in the case of diagonal noise.  Additionally, we highlight how all proof methods can be applied to obtain stronger results for stochastic functional differential equations. 
\end{abstract}
\end{frontmatter}

\section{Introduction}
Over the last four decades, a substantial literature has been developed concerning the qualitative behaviour of both deterministic and stochastic dynamical systems with memory. For systems with finite memory, the deterministic literature is summarised in classical texts such as \cite{Diekmann,KolMysh:1999(FDE)} while for the stochastic literature recent monographs by Mao and Shaikhet \cite{Mao:2008(Book),shaik:2013} have appeared. When considering deterministic Volterra equations, one can consult \cite{Cor90b,GLS}; however, to the best of our knowledge such a monograph for stochastic Volterra equations is yet to appear. In this article we endeavour to extend the results in \cite{AL:2023(AppliedMathLetters)} to the stochastic case which leads us to the following $\mathbb{R}^d$-valued stochastic Volterra integro-differential equation
\begin{equation} \label{eq. continuous SVE Intro}
    dX(t)  = \left(f(t)+\int_{[0,t]}\nu(ds)X(t-s)\right)dt+\sigma(t)dB(t), \quad t \geq 0.
\end{equation}
Here $\nu$ is understood to be an $\mathbb{R}^{d \times d}$-valued finite signed Borel measure, $\sigma$ an $\mathbb{R}^{d \times m}$-valued continuous function, $f$ an $\mathbb{R}^d$-valued continuous function, and $B$ an $\mathbb{R}^m$-valued standard Brownian motion.

The asymptotic behaviour of solutions to equations of type \eqref{eq. continuous SVE Intro} was first considered by Chan and Williams \cite{ChanWilliams:1989}. Following on from this work, an abundance of literature has appeared considering various types of behaviour both qualitative and quantitative. Asymptotic stability is considered in \cite{AP:2002(ECP),ACR:2011(DCDS),ApRie:2006(SAA)}, while subexponential and exponential asymptotic stability is explored in \cite{AP:2004(SubExpItoVol),AF:2003(EJP)}. Equations with multiplicative noise (both linear and non-linear) have also received due study \cite{AP:2021,Mao:2000(SAA),MaoRie:2006(SCL),ShaiRob:2011(DCDIS)}. Our primary goal is to provide necessary and sufficient conditions on the forcing terms $f$ and $\sigma$ in order to guarantee the trajectories of solutions to \eqref{eq. continuous SVE Intro} are almost surely $p$-integrable functions in time. Indeed, sufficient conditions are easily obtained, but to the best of our knowledge a complete characterisation has up to this point remained unknown.

The first barrier at hand is to characterise when solutions to the underlying deterministic equation (i.e., when $\sigma=0$) are elements of $L^p(\mathbb{R}_+;\mathbb{R}^d)$. This was carried out in the scalar case (with hints regarding the finite dimensional equation) by the authors in \cite{AL:2023(AppliedMathLetters)}. The condition generated does not impose any restrictions directly on $f$ but rather on a continuous linear functional applied to $f$, namely we must have for each component
\begin{equation} \label{eq. integral average lp f intro}
    t\mapsto \int_t^{t+\theta} f_i(s)ds \in L^p(\mathbb{R}_+;\mathbb{R}) \quad \text{ for all } \theta >0,
\end{equation}
to ensure solutions to the underlying deterministic equation are $p$-integrable for $p\geq1$. This extension to finite dimensions follows as a simple corollary to Theorem \ref{thm. L^p theorem continuous SVE} below. Turning our attention to $\sigma$, we see a dichotomy of cases; for $p\geq2$ one needs each component of the matrix to satisfy
\begin{equation} \label{eq. condition on sigma intro p>2}
    t \mapsto\int_{t}^{t+\theta} \sigma_{ij}^2(s)ds \in L^{\frac{p}{2}}(\mathbb{R}_+;\mathbb{R}) \quad \text{ for all } \theta>0,
\end{equation}
while for $p \in [1,2)$ we require
\begin{equation} \label{eq. condition on sigma intro p<2}
    n \mapsto\int_{n}^{n+1} \sigma_{ij}^2(s)ds \in \ell^{\frac{p}{2}}(\mathbb{N};\mathbb{R}).
\end{equation}
The border case $p=2$ arises from It\^o's isometry. If we make the standard assumption that the underlying resolvent of the measure $\nu$ is integrable\footnote{Recall this is equivalent to $\det[zI_{d\times d}-\hat{\nu}(z)]\neq 0 $ for $\text{Re}(z)\geq 0$, where $z \in \mathbb{C}$ and $\hat{\nu}$ denotes the Laplace transform of the measure $\nu$.}, then our main result shows that for $p\geq 1$ (together with any norm on $\mathbb{R}^d$), the solution to \eqref{eq. continuous SVE Intro} satisfies
\begin{equation} \label{eq. main thm intro}
   \|X(\cdot)\| \in L^p(\mathbb{R}_+;\mathbb{R})  \quad a.s. 
\end{equation}
if and only if $f$ satisfies \eqref{eq. integral average lp f intro} (componentwise) and each component of $\sigma$ obeys either \eqref{eq. condition on sigma intro p>2} or \eqref{eq. condition on sigma intro p<2}, depending on the value of $p$.

The utility of this characterisation is the fact that the perturbation functions may be very ill-behaved and yet conditions \eqref{eq. integral average lp f intro}, \eqref{eq. condition on sigma intro p>2} and \eqref{eq. condition on sigma intro p<2} may still be satisfied; some explicit examples were given by the authors in \cite{AL:2023(AppliedNumMath)} and \cite{AL:2023(AppliedMathLetters)} while a new example is provided in Section \ref{sec. Examples}. In the context of equation \eqref{eq. continuous SVE Intro}, characterising when a process has almost surely $p$-integrable sample paths is of interest due to the fact it is equivalent to the $p$-th mean of the process being integrable, i.e., \eqref{eq. main thm intro} is equivalent to $\int_0^\infty \mathbb{E}\|X(s)\|^pds < \infty$ (see Theorem \ref{thm. L^p theorem continuous SVE} below). Such a condition is often a first stop on the way to proving $\mathbb{E}\|X(t)\|^p \to 0$ as $t \to \infty$, which is of great interest in all domains where such models are employed in practice. Pairing this integrability condition with a pathwise regularity result on the mapping $t \mapsto \mathbb{E}\|X(t)\|^p$ such as Lipschitz (or uniform) continuity yields convergence to zero. For stochastic functional differential equations the classical approach is via construction of an appropriate Lyapunov functional which in turn gives the integrability of the $p$-th mean, this is outlined in \cite[Chapter 4]{KolNos:1986(SFDE)}.

The drawbacks of this method lie in the difficulty of constructing such a Lyapunov functional and the reward for such hard work is a sufficient condition for $p$-th mean integrability; necessary conditions are not easily obtained. The approach in this article is simpler and provides characterisations of $p$-integrable paths rather than generating sufficient conditions, our analysis is direct and focuses on the conditions \eqref{eq. integral average lp f intro}, \eqref{eq. condition on sigma intro p>2}, and \eqref{eq. condition on sigma intro p<2}. Additionally, we obtain precise information about the almost sure asymptotic behaviour of the paths (Theorem \ref{thm. pathwise behaviour of X}) without needing to analyse the path regularity of the $p$-th mean. In particular, we characterise almost sure convergence to zero in the special case of a diagonal diffusion matrix (Theorem \ref{thm. characterising convergence of X a.s. diagonal sigma}).

Although in this paper we are only concerned with an additive noise equation, the same methodology can be applied to study the mean square behaviour of equations with multiplicative noise where once again the use of Lyapunov functionals can be circumvented, see \cite{AL:2023(AppliedNumMath)}. To the best of our knowledge, the analysis of integrals of perturbation functions over compact intervals in the study of solutions to perturbed dynamical systems was first carried out in the papers of Strauss and Yorke \cite{SY67a,SY67b}. Our own first encounter with such objects was in the monograph by Gripenberg, Londen, and Staffans \cite[Lemma 15.9.2]{GLS}.

We make the observation that all results proven for equation \eqref{eq. continuous SVE Intro} are also true if one considers instead stochastic functional differential equations (SFDEs) of the form
\begin{align} \label{eq. SFDE Intro}
    dX(t)  &= \left(f(t)+\int_{[-\tau,0]}\mu(ds)X(t+s)\right)dt+\sigma(t)dB(t),  & t \geq 0; \nonumber\\  \quad X(t)&=\psi(t),  &t\in [-\tau,0].
\end{align}
Here, the measured-valued kernel $\mu$ is concentrated on the compact set $[-\tau,0]$ where $\tau >0$ is the fixed delay parameter and the initial data is a $C([-\tau,0];\mathbb{R}^d)$-valued random variable. We shall often write $X(t,\psi)$ when referring to solutions of \eqref{eq. SFDE Intro} to emphasize the dependence on the random initial function and to distinguish it from the solution of the Volterra equation \eqref{eq. continuous SVE Intro}. In fact, in this case we do not need any assumption regarding the integrability of the resolvent; this comes out as a necessary condition in the proof.

It transpires that in order to prove our main result, we require an auxiliary lemma regarding the summability of sequences of certain random variables which is lifted directly from the analysis of the discrete time analogue of \eqref{eq. continuous SVE Intro}. By this, we mean the $\mathbb{R}^d$-valued stochastic Volterra summation equation\footnote{We use the notation $X$ to represent solutions to both \eqref{eq. continuous SVE Intro} and \eqref{eq. Volterra difference equation Intro}, it should be clear by the context to which equation we are referring.}
\begin{equation} \label{eq. Volterra difference equation Intro}
    X(n+1) =X(n)+\sum_{j=0}^nK(n-j)X(j)+f(n) +\sigma(n)\xi(n+1), \quad n \in \mathbb{N},
\end{equation}
where $(\xi(n))_{n\geq1}$ is an independent and identically distributed (i.i.d.) sequence of $\mathbb{R}^m$-valued random vectors, $(f(n))_{n\geq0}$ and $(\sigma(n))_{n\geq0}$ are $\mathbb{R}^d$ and $\mathbb{R}^{d \times m}$-valued deterministic sequences respectively, and the kernel $(K(n))_{n\geq0}$ is a $\mathbb{R}^{d \times d}$-valued sequence with entries in $\ell^1(\mathbb{N};\mathbb{R})$.

Equations of this type have also drawn much attention from the literature. For essential background on the deterministic theory we refer the reader to \cite{Raff:2018(QVDE)}. Like in the stochastic case, there are many variations of \eqref{eq. Volterra difference equation Intro} one can study. For work focused on a fixed number of time lags, often mean square behaviour for systems with both additive and multiplicative noise is a priority \cite{AL:2023(ICDEA),Shai:1997(AML)}. The monograph by Shaikhet \cite{shaik:2011} provides an excellent overview for such equations. In the Volterra case asymptotic behaviour of the trajectories is considered for both linear deterministic equations \cite{AP:2017,Diblik:2011(AAA),GyoriReynolds:2010(periodicSolutions)}, and stochastic equations \cite{Ap-Rie-Rod:2009,BerRod:2006(JDEA)}.

Our results are of a qualitative flavour, i.e., in order for solutions of \eqref{eq. Volterra difference equation Intro} to obey $\|X\| \in \ell^p(\mathbb{N};\mathbb{R})$ almost surely, it is necessary and sufficient to have $\|f\|,\|\sigma\| \in \ell^p(\mathbb{N};\mathbb{R})$. There is a trade-off between what assumptions one makes about the noise $\xi$ and the structure one imposes on the matrix $\sigma$. For instance, one can allow the components of $\xi$ to follow essentially any arbitrary distribution (this is made precise by Definition \ref{def. large class of r.v's}) and depend on one another in any way one pleases, but this forces $\sigma$ to be diagonal. If one wishes to have a general matrix $\sigma$ then for the above result to hold we need to restrict ourselves to Gaussian sequences with independent components.

The layout of this article is as follows. In Section \ref{subsec. Notation and mathematical preliminaries} we discuss the notation and definitions necessary to study the discrete equation \eqref{eq. Volterra difference equation Intro}. Section \ref{subsec. Results for general noise} focuses on a general noise sequence without componentwise independence while Section \ref{subsec. Gaussian noise} handles the case of Gaussian noise. We then move on to study the continuous time equation, introducing notation in Section \ref{subsec. Notation}, the main results in Section \ref{subsec. Main Results}, and additional results regarding the asymptotic behaviour of sample paths in Section \ref{subsec. Almost sure asymptotic behaviour}.  In Section \ref{sec. SFDEs} we discuss stochastic functional equations before providing examples in Section \ref{sec. Examples} to emphasize the utility of theoretical results. Finally, we conclude with a discussion of results in Section \ref{sec. Conclusions}.

\section{Discrete results}

\subsection{Notation and mathematical preliminaries} \label{subsec. Notation and mathematical preliminaries}

We denote the standard basis vectors in $\mathbb{R}^d$ by $(\textbf{e}_1),\ldots,(\textbf{e}_d)$. The sum of all basis vectors in $\mathbb{R}^d$ will be denoted $\textbf{e}_{[d]}$. The standard inner product on $\mathbb{R}^d$ is denoted by $\langle \cdot,\cdot\rangle$. We say the $d\times m$ matrix-valued sequence $a:\mathbb{N} \to \mathbb{R}^{d\times m},\,  n \mapsto a(n)$ obeys $a \in \ell^p(\mathbb{N};\mathbb{R}^{d\times m})$ if each component sequence, $\left(a(n)\right)_{i,j} \in \ell^p(\mathbb{N};\mathbb{R})$ for $i=1,\ldots,d$ and $j=1,\ldots,m$, where $\left(a(n)\right)_{i,j} \coloneqq \langle \textbf{e}_i,a(n)\textbf{e}_j \rangle$. For scalar-valued objects we shall use the notation $|\cdot |$ to represent the standard euclidean norm on $\mathbb{R}$. We will use the common short hand $X \overset{\text{$d$}}{=} Y$, to denote when two random variables $X$ and $Y$ are equal in distribution and the notation $\chi$ to denote indicator random variables/functions. The abbreviation i.i.d. stands for independent and identically distributed.

We introduce the following class of real valued random variables which will be used in the proofs of a large proportion of converse results throughout this section.
\begin{definition} \label{def. large class of r.v's}
Let $\xi$ be an  $\mathbb{R}$-valued random variable. We say $\xi \in \mathbb{D}$ if there exists distinct, measurable,  and bounded Borel sets $B_1,B_2$ such that
\begin{itemize}
    \item[(i)]$\mathbb{P}(\xi \in B_i) >0$ for $i \in\{1,2\}$ with $p_1\coloneqq \mathbb{P}(\xi \in B_1)$ and $p_2 \coloneqq \mathbb{P}(\xi \in B_2)$,
    \item[(ii)] $p_2e_1-p_1e_2 \neq 0$,
\end{itemize}
where $e_i\coloneqq \mathbb{E}\xi  \chi_{\{\xi \in B_i\}}$ for $i \in \{1,2\}$.
\end{definition}

\begin{remark}
    The class of random variables introduced in Definition \ref{def. large class of r.v's} is certainly nonstandard. The definition is made in this manner mainly for technical reasons that appear in subsequent proofs, however upon reflection one realises that this class of random variables is very large indeed. It is the authors' opinion that after excluding random variables with singular laws, the only random variables that are not elements of the space $\mathbb{D}$ are the almost surely constant random variables. No proof is provided for this conjecture. Instead, we prove that if a random variable $\xi$ has a distribution with either an absolutely continuous part or at least two isolated atoms then $\xi \in \mathbb{D}$, see Proposition \ref{prop. examples of random variables} below, the proof of which is relegated to the appendix. 
\end{remark}
\begin{prop} \label{prop. examples of random variables}
Let $\xi$ be an $\mathbb{R}$-valued random variable whose distribution has either an absolutely continuous part or at least two isolated atoms (with non-zero probability), then $\xi \in \mathbb{D}$.
\end{prop}
\subsection{Results for general noise} \label{subsec. Results for general noise}

Let $(\Omega, \mathcal{F},\mathbb{P})$ be a complete probability space. We study the following $\mathbb{R}^d$-valued stochastic Volterra summation equation

\begin{equation} \label{eq. Volterra difference equation}
    X(n+1) =X(n)+\sum_{j=0}^nK(n-j)X(j)+f(n) +\sigma(n)\xi(n+1), \quad n \in \mathbb{N}; \quad X(0) = x. 
\end{equation}
Here $x \in \mathbb{R}^d$, $K$ and $\sigma$ are deterministic $\mathbb{R}^{d \times d}$-valued sequences, and $f$ is a deterministic $\mathbb{R}^d$-valued sequence. Each $\xi(n)$ is a random vector in $\mathbb{R}^d$. We make the following standing assumptions throughout this subsection:
\begin{align}
    & (\sigma(n))_{n\geq0} \text{ is a sequence of diagonal matrices,} \label{assum. sigma(n) is diagonal}\\
    & (\xi(n))_{n\geq1} \text{ is an independent and identically distributed sequence of random vectors,} \label{assum. xi is an i.i.d sequence of r.v}\\
    & \text{For all } n \in \mathbb{Z}_+,\, \langle \xi(n), \textbf{e}_j\rangle \overset{\text{$d$}}{=} \eta_j  \text{ where } \eta_j \in \mathbb{D}, \text{ for } j=1,\ldots,d. \label{assum. each component of xi(n) is equal in distribution to xi in D}
\end{align}
We assume our probability space is rich enough to support the sequence of random vectors $(\xi(n))_{n\geq1}$ and the random vector $\eta$  and equip it with the natural filtration generated by $\xi(n)$, i.e.,  $\mathcal{F}_n=\sigma\left(\xi(k): 1\leq k \leq n \right)$ with the convention that $\mathcal{F}_0=\{\emptyset,\mathcal{B},\Omega\}$. Here $\mathcal{B}$ denotes the sigma algebra generated by $\eta$. It is worth noting some common assumptions that we are \emph{not} making. We do not require the random variables $\langle \xi(n), \textbf{e}_j\rangle$ and $\langle \xi(n), \textbf{e}_i\rangle$ for $i\neq j$ to be independent, nor do we require that they have the same distribution. We now introduce an auxiliary object which allows efficient study of equation \eqref{eq. Volterra difference equation}. Let the so-called resolvent be the $\mathbb{R}^{d\times d}$-valued solution to the following matrix equation
\begin{equation} \label{eq. Volterra difference resolvent}
    R(n+1) =R(n)+\sum_{j=0}^nK(n-j)R(j), \quad n \in \mathbb{N}; \quad R(0) =I_{d\times d}. 
\end{equation}
Thus, the well known variation of constants formula gives
\begin{equation} \label{eq. Volterra difference variation of constants}
    X(n) =R(n)X(0)+\sum_{j=1}^nR(n-j)\left[f(j-1) +\sigma(j-1)\xi(j)\right], \quad n \in \mathbb{Z}_+. 
\end{equation}
In the sequel we will always assume
\begin{equation} \label{assum. summability of kernel and resolvent}
K \in \ell^1(\mathbb{N};\mathbb{R}^{d\times d}); \quad R \in \ell^1(\mathbb{N};\mathbb{R}^{d\times d}).   
\end{equation}
This is a standard assumption one needs to make in order to infer any reasonable convergence results for solutions to \eqref{eq. Volterra difference equation}, see \cite{Ap-Rie-Rod:2009}.

\begin{theorem} \label{thm. lp characterisation of Volterra difference equation}
Let $p\in [1,\infty)$. Suppose $(\xi(n))_{n\geq1}$ obeys \eqref{assum. xi is an i.i.d sequence of r.v}-\eqref{assum. each component of xi(n) is equal in distribution to xi in D} and additionally that $\langle \xi(n),\textnormal{$\textbf{e}_j$} \rangle \in L^p(\Omega)$ for $j=1,\ldots,d$. Let $K$ and $R$ obey \eqref{assum. summability of kernel and resolvent}, $\sigma$ obey \eqref{assum. sigma(n) is diagonal}, and let $X$ be the solution of \eqref{eq. Volterra difference equation}. Then the following are equivalent:

\begin{itemize}
    \item[(i)] $f \in \ell^p(\mathbb{N},\mathbb{R}^d)$ and $\sigma \in \ell^p(\mathbb{N};\mathbb{R}^{d \times d})$,
    \item[(ii)] $X \in \ell^p(\mathbb{N},\mathbb{R}^d)$ almost surely.
\end{itemize}
\end{theorem}
An interesting question to ask is whether $p$-summability is ever possible for the solution of \eqref{eq. Volterra difference equation} without the deterministic perturbations $f$ and $\sigma$ being $p$-summable. One might conjecture that this is possible but, upon inspection of the proof of the converse result of Theorem \ref{thm. lp characterisation of Volterra difference equation}, we see that this cannot be the case. We state this explicitly as a corollary whose proof follows directly from the proof of Theorem \ref{thm. lp characterisation of Volterra difference equation}.

\begin{corollary} \label{cor. lp summability converse}
Suppose $(\xi(n))_{n\geq1}$ obeys \eqref{assum. xi is an i.i.d sequence of r.v}-\eqref{assum. each component of xi(n) is equal in distribution to xi in D}. Let $K$ and $R$ obey \eqref{assum. summability of kernel and resolvent}, $\sigma$ obey \eqref{assum. sigma(n) is diagonal}, let $X$ be the solution of \eqref{eq. Volterra difference equation}, and $p \in [1,\infty)$. If $X \in \ell^p(\mathbb{N},\mathbb{R}^d)$ a.s. then $f \in \ell^p(\mathbb{N},\mathbb{R}^d)$ and $\sigma \in \ell^p(\mathbb{N};\mathbb{R}^{d \times d})$.
\end{corollary}

Corollary \ref{cor. lp summability converse} implies that a ``\emph{stabilisation by noise}" effect cannot occur. It seems the only way such a phenomenon can occur is if one relaxes the identically distributed assumption on the noise sequence to allow the possibility of $\xi(n) \to 0$ in an appropriate sense. Before proving these results we prepare a lemma.

\begin{lemma} \label{lem. summability converse}
    Let $(f(n))_{n\geq0}$ and $(\sigma(n))_{n\geq0}$ be deterministic scalar sequences and $(\xi(n))_{n\geq1}$ be an i.i.d. sequence of scalar random variables with $\xi(n)\overset{\text{$d$}}{=} \eta$ where $\eta \in \mathbb{D}$. If for any $p \in [1,\infty)$
    \begin{equation} \label{eq. p summability of perturbation}
        \sum_{n=0}^\infty \left|f(n)+\sigma(n)\xi(n+1) \right|^p < \infty \quad \text{a.s.},
    \end{equation}
    then $f,\sigma \in \ell^p(\mathbb{N};\mathbb{R})$.
\end{lemma}

\begin{proof}[Proof of Lemma \ref{lem. summability converse}]
First we show that \eqref{eq. p summability of perturbation} implies both $f$ and $\sigma$ are null sequences and hence bounded. We note that \eqref{eq. p summability of perturbation} implies that $f(n)+\sigma(n)\eta  \overset{n \to \infty}{\longrightarrow} 0$ almost surely, to see this consider the following argument.

By \eqref{eq. p summability of perturbation} we have for all $\varepsilon > 0$, $\mathbb{P}\left[ |f(n)+\sigma(n)\xi(n+1)|> \varepsilon, \text{ i.o.}\right]=0$. Now fix an $\varepsilon >0 $ and suppose $\mathbb{P}\left[ |f(n)+\sigma(n)\eta|> \varepsilon, \text{ i.o.}\right]>0$. The first Borel-Cantelli lemma implies
\[
\sum_{n=0}^\infty \mathbb{P}\left[ |f(n)+\sigma(n)\eta|> \varepsilon \right] = \infty,
\]
which by the equality in distribution of $\xi(n)$ and $\eta$, forces
\[
\sum_{n=0}^\infty \mathbb{P}\left[ |f(n)+\sigma(n)\xi(n+1)|> \varepsilon \right] = \infty.
\]
Then leveraging the independence of the $\xi(n)$'s, we apply the second Borel-Cantelli lemma to give
\[
\mathbb{P}\left[ |f(n)+\sigma(n)\xi(n+1)|> \varepsilon, \text{ i.o.}\right]=1,
\]
which is a contradiction and so $\mathbb{P}\left[ |f(n)+\sigma(n)\eta|> \varepsilon, \text{ i.o.}\right]=0$. As $\varepsilon$ was arbitrary, this yields $f(n)+\sigma(n)\eta  \overset{n \to \infty}{\longrightarrow} 0$ almost surely. Fix $\omega_1$ and $\omega_2$ such that $\eta(\omega_1)\neq \eta(\omega_2)$, which can always be done as $\eta \in \mathbb{D}$. Thus,
\begin{align*}
    (f(n)+\sigma(n)\eta(\omega_1))-(f(n)+\sigma(n)\eta(\omega_2)) \to 0,
\end{align*}
which yields
\begin{equation*}
    \sigma(n)(\eta(\omega_1)-\eta(\omega_2)) \to 0.
\end{equation*}
As $\eta(\omega_1)-\eta(\omega_2) \neq 0$, this implies $\sigma$ is a null sequence which forces $f$ to be a null sequence and so they are both bounded. Now let $B_1$ and $B_2$ be the two Borel sets which come from Definition \ref{def. large class of r.v's} and introduce the sequences
\begin{align*}
    X_n & \coloneqq \left|f(n)+\sigma(n)\xi(n+1) \right|^p  \chi_{\{\xi(n+1) \in B_1\}},\\
    Y_n & \coloneqq \left|f(n)+\sigma(n)\xi(n+1) \right|^p  \chi_{\{\xi(n+1) \in B_2\}}.
\end{align*}
The boundedness of $f$ and $\sigma$ along with the truncation of the noise sequence $\xi(n)$ means $X_n$ and $Y_n$ are in fact almost surely uniformly bounded random variables. Clearly from \eqref{eq. p summability of perturbation} we have that $X_n, Y_n \in \ell^1(\mathbb{N})$ almost surely, this along with uniform boundedness allows us to apply Kolmogorov's two-series test which yields
\begin{equation*}
    \sum_{n=0}^\infty\mathbb{E}X_n < \infty, \quad \sum_{n=0}^\infty\mathbb{E}Y_n < \infty.
\end{equation*}
Now focusing on the first series we see
\begin{align*}
    \infty & > \sum_{n=0}^\infty\mathbb{E}\left[\left|f(n)+\sigma(n)\xi(n+1) \right|^p  \chi_{\{\xi(n+1) \in B_1\}}\right]\\
    & = \sum_{n=0}^\infty\mathbb{E}\left[\left|f(n)+\sigma(n)\eta \right|^p  \chi_{\{\eta \in B_1\}}\right]\\
    & \geq \sum_{n=0}^\infty\left| \mathbb{E}\left[f(n) \chi_{\{\eta \in B_1\}}+\sigma(n)\eta  \chi_{\{\eta \in B_1\}}\right]\right|^p\\
    & = \sum_{n=0}^\infty \left|p_1f(n)+e_1\sigma(n) \right|^p,
\end{align*}
where the inequality comes from Jensen as $p\geq 1$ and $p_1,e_1$ are constants from Definition \ref{def. large class of r.v's}. An identical argument using $\sum_{n=0}^\infty\mathbb{E}Y_n$ also yields
\begin{equation*}
    \sum_{n=0}^\infty \left|p_2f(n)+e_2\sigma(n) \right|^p < + \infty.
\end{equation*}
Next we leverage the fact that $p_2e_1-p_1e_2 \neq 0$. Consider
\begin{align*}
    \sigma(n)(p_2e_1-p_1e_2) & = \sigma(n)(p_2e_1-p_1e_2)+p_1p_2f(n)-p_1p_2f(n)\\
    & = p_2(p_1f(n)+e_1\sigma(n))-p_1(p_2f(n)+e_2\sigma(n)).
\end{align*}
Thus,
\begin{equation*}
    |p_2e_1-p_1e_2|^p \sum_{n=0}^\infty |\sigma(n)|^p \leq |p_2|^p\sum_{n=0}^\infty\left|p_1f(n)+e_1\sigma(n) \right|^p + |p_1|^p \sum_{n=0}^\infty \left|p_2f(n)+e_2\sigma(n) \right|^p.
\end{equation*}
As both series on the right are finite we have $\sigma \in \ell^p(\mathbb{N};\mathbb{R})$. A similar argument also shows $f \in \ell^p(\mathbb{N};\mathbb{R})$ and the lemma is proven.
\end{proof}

\begin{proof}[Proof of Theorem \ref{thm. lp characterisation of Volterra difference equation}]
To prove the forward implication it is enough to use \eqref{eq. Volterra difference variation of constants} and show $f(n)+\sigma(n)\xi(n+1)$ is an element of $\ell^p(\mathbb{N};\mathbb{R}^d)$ almost surely, as $\ell^p(\mathbb{N};\mathbb{R}^d)$ is closed under convolution with $\ell^1(\mathbb{N};\mathbb{R}^{d\times d})$ sequences (and $R \in \ell^1(\mathbb{N};\mathbb{R}^{d\times d})$ by Assumption \eqref{assum. summability of kernel and resolvent}). This amounts to showing the component sequence $\langle f(n)+\sigma(n)\xi(n+1),\textbf{e}_i \rangle \in \ell^p(\mathbb{N};\mathbb{R})$ almost surely for an arbitrary $i \in \{1,\ldots,d\}$. The triangle inequality and $\langle \xi(n),\textbf{e}_i \rangle \in L^p(\Omega)$ yield
\[
\mathbb{E}\left[\left| \langle f(n)+\sigma(n)\xi(n+1),\textbf{e}_i \rangle \right|^p\right]=\mathbb{E}\left|f_i(n)+\sigma_{ii}(n)\xi_i(n+1) \right|^p \leq \left|f_i(n)\right|^p+C_p\left|\sigma_{ii}(n)\right|^p.
\]
Invoking $(i)$ yields
\[
\sum_{n=0}^\infty \mathbb{E}\left[\left| \langle f(n)+\sigma(n)\xi(n+1),\textbf{e}_i \rangle \right|^p\right] < \infty,
\]
which implies
\begin{equation*}
  \sum_{n=0}^\infty \left| \langle f(n)+\sigma(n)\xi(n+1),\textbf{e}_i \rangle \right|^p < \infty \quad \text{a.s}.  
\end{equation*}
For the reverse implication we use the identity
\[
X(n+1)-X(n)-\sum_{j=0}^nK(n-j)X(j) = f(n) +\sigma(n)\xi(n+1).
\]
As $K \in \ell^1(\mathbb{N};\mathbb{R}^{d\times d})$ and $X \in \ell^p(\mathbb{N};\mathbb{R}^{d})$ almost surely, the right hand side must also be an element of $\ell^p(\mathbb{N};\mathbb{R}^{d})$ almost surely. Thus, for each $i \in \{1,\ldots,d\}$
\[
\sum_{n=0}^\infty \left|f_i(n)+\sigma_{ii}(n)\xi_i(n+1) \right|^p < \infty \quad \text{ a.s}.
\]
Thus, Lemma \ref{lem. summability converse} implies $f_i, \sigma_{ii} \in \ell^p(\mathbb{N};\mathbb{R})$ and so $f \in \ell^p(\mathbb{N},\mathbb{R}^d)$ and $\sigma \in \ell^p(\mathbb{N};\mathbb{R}^{d \times d})$.
    
\end{proof}

\subsection{Gaussian noise} \label{subsec. Gaussian noise}
If we specialise to the case where $(\xi(n))_{n\geq1}$ is an i.i.d. sequence of Gaussian random vectors with zero mean and independent components, then we can relax the assumption of $\sigma$ being a diagonal matrix. From now on $\sigma$ is an $\mathbb{R}^{d\times m}$ valued sequence and the noise vector $\xi$ will take values in $\mathbb{R}^m$ for some $m\geq d$.  Introduce the following assumptions

\begin{align}
    & (\xi(n))_{n\geq1} \text{ is an i.i.d sequence of } \mathbb{R}^m \text{-valued Gaussian random vectors with zero mean.} \label{assum. xi is an i.i.d sequence of r.v gaussian case}\\
    & \text{For all } n \in \mathbb{N},\,\langle \xi(n),\textbf{e}_i \rangle \text{ is independent of } \langle \xi(n),\textbf{e}_j \rangle \text{ for all } i\neq j. \label{assum. each component of xi(n) is independent gaussian case}
\end{align}

\begin{theorem} \label{thm. lp characterisation of Volterra difference equation gaussian noise case}
Let $p \in [1,\infty)$. Suppose $(\xi(n))_{n\geq1}$ obeys \eqref{assum. xi is an i.i.d sequence of r.v gaussian case}-\eqref{assum. each component of xi(n) is independent gaussian case}, $K$ and $R$ obey \eqref{assum. summability of kernel and resolvent}, and $X$ is the solution of \eqref{eq. Volterra difference equation}. Then the following are equivalent:

\begin{itemize}
    \item[(i)] $f \in \ell^p(\mathbb{N},\mathbb{R}^d)$ and $\sigma \in \ell^p(\mathbb{N};\mathbb{R}^{d \times m})$,
    \item[(ii)] $X \in \ell^p(\mathbb{N},\mathbb{R}^d)$ almost surely.
\end{itemize}
\end{theorem}
For a proof of Theorem \ref{thm. lp characterisation of Volterra difference equation gaussian noise case} follow the proof of Theorem \ref{thm. lp characterisation of Volterra difference equation} verbatim, the only difference being in the converse invoke Lemma \ref{lem. summability converse with gaussian noise} instead of Lemma \ref{lem. summability converse}.

\begin{lemma} \label{lem. summability converse with gaussian noise}
  Let $(f(n))_{n\geq0}$ and $(\sigma(n))_{n\geq0}$ be deterministic, $\mathbb{R}^d$ and $\mathbb{R}^{d \times m}$-valued sequences respectively. Let $(\xi(n))_{n\geq1}$ obey \eqref{assum. xi is an i.i.d sequence of r.v gaussian case}-\eqref{assum. each component of xi(n) is independent gaussian case}. Then for $p \in [1,\infty)$, if
  \[
  f+\sigma \xi \in \ell^p(\mathbb{N};\mathbb{R}^d) \quad \text{ a.s.}, 
  \]
  then $f \in \ell^p(\mathbb{N};\mathbb{R}^d)$ and $\sigma \in \ell^p(\mathbb{N};\mathbb{R}^{d \times m})$.
\end{lemma}

\begin{proof}[Proof of Lemma \ref{lem. summability converse with gaussian noise}]
For each $i \in \{1,\ldots,d\}$
\[
\sum_{n=0}^\infty \left|f_i(n)+\sum_{j=1}^m \sigma_{ij}(n)\xi_j(n+1)\right|^p < \infty \quad \text{ a.s}.
\]
Define
\[
\zeta_i(n) \coloneqq \sum_{j=1}^m \sigma_{ij}(n)\xi_j(n+1).
\]
Then $\zeta_i(n)$ is a sequence of independent and normally distributed random variables with zero mean. Now consider
\[
\zeta_i(n)=\sqrt{\mathbb{E}|\zeta_i(n)|^2}  \dfrac{\zeta_i(n)}{\sqrt{\mathbb{E}|\zeta_i(n)|^2}}=\sqrt{\mathbb{E}|\zeta_i(n)|^2}\, \tilde{\zeta}_i(n),
\]
where $\tilde{\zeta}_i(n)$ is an i.i.d sequence of standard normal random variables\footnote{We need not worry that $\sqrt{\mathbb{E}|\zeta_i(n)|^2}=0$, as if this is the case then $\zeta_i(n)=0$ almost surely and so this term will have no contribution to the series. Thus, in this instance we define $\tilde{\zeta}_i(n)\coloneqq\zeta_i(n)$.}. The series then becomes
\[
\sum_{n=0}^\infty \left|f_i(n)+\sqrt{\mathbb{E}|\zeta_i(n)|^2} \tilde{\zeta}_i(n)\right|^p < \infty \quad \text{ a.s}.
\]
Invoking Lemma \ref{lem. summability converse} yields
\[
\sum_{n=0}^\infty \left|f_i(n)\right|^p + \sum_{n=0}^\infty\left|\sqrt{\mathbb{E}|\zeta_i(n)|^2}\right|^p < \infty.
\]
Recall that
\[
\mathbb{E}|\zeta_i(n)|^2=\sum_{j=1}^m\sigma_{ij}^2(n)\mathbb{E}|\xi_j(n+1)|^2=\sum_{j=1}^mC_j  \sigma_{ij}^2(n).
\]
Thus,
\[
\infty \geq \sum_{n=0}^\infty\left|\sqrt{\mathbb{E}|\zeta_i(n)|^2}\right|^p =  \sum_{n=0}^\infty \left(\sum_{j=1}^mC_j \sigma_{ij}^2(n)\right)^{p/2} \geq C\sum_{n=0}^\infty \sum_{j=1}^m |\sigma_{ij}(n)|^p,
\]
for some $C>0$. The last inequality follows from the fact that norms on finite dimensional vector spaces are equivalent and $p\geq1$. Thus, $f_i,\sigma_{ij} \in \ell^p(\mathbb{N};\mathbb{R})$ for $i \in \{1,\ldots,d\}, j \in \{1,\ldots,m\}$. The claim is proven.
    
\end{proof}
\section{Continuous Results} \label{sec. Continuous results}

\subsection{Notation} \label{subsec. Notation}
We say a function $f:\mathbb{R}_+\to \mathbb{R}^{d \times m}$ obeys $f \in L^p(\mathbb{R}_+;\mathbb{R}^{d \times m})$ if each component function $(f(t))_{i,j} \in L^p(\mathbb{R}_+;\mathbb{R})$ for $i=1,\ldots,d$ and $j=1,\ldots,m$, where $\left(f(t)\right)_{i,j} \coloneqq \langle \textbf{e}_i,f(t)\textbf{e}_j \rangle$. We make this non-standard choice of notation to emphasize the fact that all proofs are carried out componentwise. This definition is equivalent to $\|f\| \in L^p(\mathbb{R}_+;\mathbb{R})$ for any chosen norm on $\mathbb{R}^{d\times m}$. Let $M(E;\mathbb{R}^{d\times d})$ denote the space of $d \times d$ matrix valued, finite, signed Borel measures on $E\subset \mathbb{R}$. Given $\nu \in M(E;\mathbb{R}^{d\times d})$ the total variation measure of each component is denoted $|\nu_{ij}|$. The space of continuous functions $f:\mathbb{R}_+ \to \mathbb{R}^{d\times m}$ is denoted by $C(\mathbb{R}_+;\mathbb{R}^{d\times m})$, while $BC_0(\mathbb{R}_+;\mathbb{R}^{d\times m})$ is the space of bounded continuous functions which vanish at infinity. If $X$ is an $\mathbb{R}^d$-valued stochastic process then once again we note the statement $\mathbb{E}|X_i(\cdot)|^p \in L^p(\mathbb{R}_+;\mathbb{R})$ for each $i \in \{1,\ldots,d\}$ is equivalent to  $\mathbb{E}\|X(\cdot)\|^p \in L^p(\mathbb{R}_+;\mathbb{R})$ for any chosen norm on $\mathbb{R}^d$. When there is no scope for ambiguity to arise, $\|\cdot \|$ will denote an arbitrary norm on the given finite dimensional vector space under consideration. For scalar-valued objects we shall use the notation $|\cdot |$ to represent the standard euclidean norm on $\mathbb{R}$. If $a,b \in \mathbb{R}$, $a\wedge b$ and $a\vee b$ denote $\min$ and $\max$ of the two numbers respectively.

\subsection{Main Results} \label{subsec. Main Results}
Consider a complete filtered probability space $(\Omega,\mathcal{F},\mathcal{F}_t,\mathbb{P})$ supporting an $m$-dimensional standard Brownian motion, $(B_t)_{t\geq0}$. We study the following $\mathbb{R}^d$-valued stochastic Volterra integro-differential equation

\begin{equation} \label{eq. continuous SVE}
    dX(t)  = \left(f(t)+\int_{[0,t]}\nu(ds)X(t-s)\right)dt+\sigma(t)dB(t), \quad t \geq 0; \quad    X(0)  = x \in \mathbb{R}^d.
\end{equation}
Here $f \in C(\mathbb{R}_+;\mathbb{R}^d)$, $\nu \in M(\mathbb{R}_+;\mathbb{R}^{d\times d})$ and $\sigma \in C(\mathbb{R}_+;\mathbb{R}^{d\times m})$. Note that, as $B$ is a standard  $m$-dimensional Brownian motion, $ \langle B(t),\textbf{e}_i \rangle$ and $ \langle B(t),\textbf{e}_j \rangle$ are independent for all $i\neq j$. We write precisely the conditions on the forcing functions $f$ and $\sigma$ which we shall be primarily concerned with throughout the rest of this paper.

\begin{align}
    & \text{For } p \in [1,\infty), \quad \int_{\cdot}^{\cdot+\theta} f_i(s)ds \in L^p(\mathbb{R}_+;\mathbb{R}) \quad \forall\, \theta>0, \quad  i \in \{1,\ldots,d\}, \label{cond. f}\\
    & \text{For } p \in [2,\infty), \int_{\cdot}^{\cdot+\theta} \sigma_{ij}^2(s)ds \in L^{\frac{p}{2}}(\mathbb{R}_+;\mathbb{R})\quad \forall\,  \theta>0,\,\,  i \in \{1,\ldots,d\},\,j \in \{1,\ldots,m\}, \label{cond. sigma p geq 2}\\
    & \text{For } p \in [1,2), \int_{\cdot}^{\cdot+1} \sigma_{ij}^2(s)ds \in \ell^{\frac{p}{2}}(\mathbb{N};\mathbb{R} ), \quad i \in \{1,\ldots,d\},\,j \in \{1,\ldots,m\}. \label{cond. sigma p<2}
\end{align}

Next we introduce the continuous time counterpart of \eqref{eq. Volterra difference resolvent}. The differential resolvent for \eqref{eq. continuous SVE} is the $d\times d$ matrix-valued solution to the equation

\begin{equation} \label{eq. continuous resolvent}
  \dot{r}(t)=\int_{[0,t]}\nu(ds)r(t-s), \quad t>0; \quad r(0)=I_{d\times d}. 
\end{equation}
In the sequel we shall employ the following standing assumption

\begin{align} \label{assum. continuous resolvent is L1}
    r \in L^1(\mathbb{R}_+;\mathbb{R}^{d\times d}).
\end{align}
For the reader's convenience we recall a classical characterisation of condition \eqref{assum. continuous resolvent is L1}. If $\nu \in M(\mathbb{R}_+;\mathbb{R}^{d \times d})$ then
\begin{align*}
    r \in L^1(\mathbb{R}_+;\mathbb{R}^{d\times d}) \iff \text{det}[zI_{d \times d}+ \int_{\mathbb{R}_+}e^{-zt}\nu(dt)] \neq 0, \quad \text{Re}(z)\geq 0.
\end{align*}
For a thorough study of equation \eqref{eq. continuous resolvent} we refer the reader to \cite{GLS}. With preliminaries dispensed with we now present our main result.

\begin{theorem} \label{thm. L^p theorem continuous SVE}
    Let $p\in [1,\infty)$, $r$ obey \eqref{assum. continuous resolvent is L1}, and $X$ be the solution of  \eqref{eq. continuous SVE}. Then we have the following dichotomy:
    \begin{enumerate}
        \item[(\textbf{A})] If $p \in [2,\infty)$, the following are equivalent:
    \begin{itemize}
        \item[(i)] $\mathbb{E}\|X(\cdot)\|^p \in L^1(\mathbb{R}_+;\mathbb{R})$,
        \item[(ii)]$X(\cdot) \in L^p(\mathbb{R}_+;\mathbb{R}^d)$ almost surely,
        \item[(iii)] $f$ obeys \eqref{cond. f} and $\sigma$ obeys \eqref{cond. sigma p geq 2}.
    \end{itemize}

    \item[(\textbf{B})] If $p \in [1,2)$, the following are equivalent:
    \begin{itemize}
        \item[(i)] $\mathbb{E}\|X(\cdot)\|^p \in L^1(\mathbb{R}_+;\mathbb{R})$,
        \item[(ii)]$X(\cdot) \in L^p(\mathbb{R}_+;\mathbb{R}^d)$ almost surely,
        \item[(iii)]$f$ obeys \eqref{cond. f} and $\sigma$ obeys \eqref{cond. sigma p<2}.
    \end{itemize}
    \end{enumerate}
\end{theorem}
    
Before we discuss the proof of Theorem \ref{thm. L^p theorem continuous SVE} we mention the assumption $r \in L^1(\mathbb{R}_+;\mathbb{R}^{d\times d})$ is in a certain sense necessary. By this we mean a converse result of the type given by Theorem 3.9 in \cite{GLS} also holds in the stochastic case even though its proof is essentially a trivial consequence of the deterministic theorem. Nonetheless we state it here for completeness.

\begin{theorem} \label{thm. SVE r L1 converse}
Let $p\in [1,\infty)$ and $X$ be the solution of  \eqref{eq. continuous SVE}. Then the following statements are true:
     \begin{enumerate}
        \item[(\textbf{A})] If $p \in [2,\infty)$, the following are equivalent:
    
    \begin{itemize}
        \item[(i)] $r \in L^1(\mathbb{R}_+;\mathbb{R}^{d\times d})$,
        \item[(ii)]  For all $x \in \mathbb{R}^d$, $f$ and $\sigma$ satisfying \eqref{cond. f} and \eqref{cond. sigma p geq 2} respectively, we have\\ $\mathbb{E}\|X(\cdot)\|^p \in L^1(\mathbb{R}_+;\mathbb{R})$,
        \item[(iii)]   For all $x \in \mathbb{R}^d$, $f$ and $\sigma$ satisfying \eqref{cond. f} and \eqref{cond. sigma p geq 2} respectively, we have \\$X(\cdot) \in L^p(\mathbb{R}_+;\mathbb{R}^d)$ almost surely.
    \end{itemize}

    \item [(\textbf{B})] If $p \in [1,2)$, the following are equivalent:
    \begin{itemize}
        \item[(i)]$r \in L^1(\mathbb{R}_+;\mathbb{R}^{d\times d})$,
        \item[(ii)]  For all $x \in \mathbb{R}^d$, $f$ and $\sigma$ satisfying \eqref{cond. f} and \eqref{cond. sigma p<2} respectively, we have\\ $\mathbb{E}\|X(\cdot)\|^p \in L^1(\mathbb{R}_+;\mathbb{R})$,
        \item[(iii)]   For all $x \in \mathbb{R}^d$, $f$ and $\sigma$ satisfying \eqref{cond. f} and \eqref{cond. sigma p<2} respectively, we have\\ $X(\cdot) \in L^p(\mathbb{R}_+;\mathbb{R}^d)$ almost surely.
    \end{itemize}
    \end{enumerate}
\end{theorem}
\begin{proof}[Proof of Theorem \ref{thm. SVE r L1 converse}]
For both statements (\textbf{A}) and (\textbf{B}) the fact that (i) $\implies$ (ii) \& (iii) follows directly from Theorem \ref{thm. L^p theorem continuous SVE}, hence we need only show (iii) $\implies$ (i). We prove this for (\textbf{A}) and (\textbf{B}) simultaneously by setting $\sigma = 0$, hence conditions \eqref{cond. sigma p geq 2} and \eqref{cond. sigma p<2} are automatically satisfied. Now fix $f \in L^p(\mathbb{R}_+;\mathbb{R}^d)$; each component of $f$ must also satisfy \eqref{cond. f}, this follows from Theorem \ref{thm. L^p theorem continuous SVE} with $\nu(dt)=-\delta_{0}(dt)I_{d\times d}$ where $\delta_0$ is a point mass measure at zero. The claim then follows from \cite[Theorem 3.9]{GLS}.
\end{proof}
Although Theorem \ref{thm. SVE r L1 converse} is not entirely satisfactory it seems one cannot do better. One would ideally like to prove that $r \in L^1(\mathbb{R}_+;\mathbb{R}^{d\times d})$ starting from the hypothesis that for all initial conditions $x \in \mathbb{R}^d$ the solution obeyed $X(\cdot) \in L^p(\mathbb{R}_+;\mathbb{R}^d)$ almost surely. However, even in the deterministic case (i.e., with $\sigma=0$) such a result is unavailable. One case in which a much stronger result can be proven is that of finite memory, this is achieved by Theorem \ref{thm. L^p theorem for continuous SFDE} in Section \ref{sec. SFDEs}. In this situation one relies heavily on an asymptotic expansion for the resolvent which is far more explicit than what one would get in the Volterra setting.

Next we turn to the task of proving Theorem \ref{thm. L^p theorem continuous SVE}. In order to efficiently study equation \eqref{eq. continuous SVE} we introduce a simpler equation which is embedded in \eqref{eq. continuous SVE}, namely if one sets $\nu(dt)=-\delta_{0}(dt)I_{d\times d}$ then
\begin{equation} \label{eq. continuous OU process}
    dY(t)  = \left(f(t)-Y(t)\right)dt+\sigma(t)dB(t), \quad t \geq 0; \quad    Y(0)  = 0.
\end{equation}
Next we introduce a lemma which shows the $p$-integrability of solutions to \eqref{eq. continuous SVE} is equivalent to the $p$-integrability of solutions to \eqref{eq. continuous OU process}. This is a very convenient result as \eqref{eq. continuous OU process} admits an explicit solution and is much more tractable from an analytic point of view.

\begin{lemma} \label{lem. L^p of SVE and OU are equivalent}
Let $p\in [1,\infty)$, $r$ obey \eqref{assum. continuous resolvent is L1}, $X$ be the solution of  \eqref{eq. continuous SVE}, and $Y$ be the solution of \eqref{eq. continuous OU process}. Then
\begin{itemize}
    \item[(i)] $\mathbb{E}\|X(\cdot)\|^p \in L^1(\mathbb{R}_+;\mathbb{R}) \iff \mathbb{E}\|Y(\cdot)\|^p \in L^1(\mathbb{R}_+;\mathbb{R})$,
    \item[(ii)] $X(\cdot) \in L^p(\mathbb{R}_+;\mathbb{R}^d)$ almost surely $\iff$ $Y(\cdot) \in L^p(\mathbb{R}_+;\mathbb{R}^d)$ almost surely.
\end{itemize}
\end{lemma}
\begin{remark}
    Lemma \ref{lem. L^p of SVE and OU are equivalent} has deep consequences. Essentially this lemma tells us that the path dependence of the Volterra equation has no bearing whatsoever on whether the paths are integrable functions in time, this is purely dictated by the perturbation functions. A priori there seems no reason to suspect this is the case, this result suggests in order for the path dependence to have significant impact on the qualitative behaviour of the solution one needs stronger memory, i.e., one could consider instead sigma finite measures for the kernel. 
\end{remark}

\begin{proof}[Proof of Lemma \ref{lem. L^p of SVE and OU are equivalent}]
    Suppose $\mathbb{E}\|X(\cdot)\|^p \in L^1(\mathbb{R}_+;\mathbb{R})$. Rewrite \eqref{eq. continuous SVE} as
    \[
    dX(t)=\left(f(t)+Q(t)-X(t)\right)dt+\sigma(t)dB(t),
    \]
    where $Q(t)=X(t)+\int_{[0,t]}\nu(ds)X(t-s)$. For the $i^{th}$ component of $Q$ we have using standard estimates
    \[
    \mathbb{E}|Q_i(t)|^p \leq C\mathbb{E}|X_i(t)|^p+C\sum_{j=1}^d \left(|\nu_{ij}| \ast \mathbb{E}|X_j(\cdot)|^p\right)(t),
    \]
    where $C>0$. Then by our supposition and the fact $|\nu_{ij}|$ is a finite measure, the right hand side must be integrable and so $\mathbb{E}|Q_i(t)|^p$ is integrable. Next we write down $X$ in terms of $Y$ and $Q$. The $i^{th}$ component of $X(t)$ is
    \begin{align*}
        X_i(t)& =x_i e^{-t}+\int_0^t e^{-(t-s)}(f_i(s)+Q_i(s))ds+ \sum_{j=1}^m\int_0^te^{-(t-s)}\sigma_{ij}(s)dB_{j}(s)\\
        & = x_ie^{-t} + \int_0^t e^{-(t-s)}Q_i(s)ds + Y_i(t).
    \end{align*}
    Rearranging yields
    \[
    Y_i(t)=X_i(t)-x_ie^{-t}-\int_0^t e^{-(t-s)}Q_i(s)ds.
    \]
    This representation for $Y_i(t)$ gives the estimate
    \[
    \mathbb{E}|Y_i(t)|^p \leq C\mathbb{E}|X_i(t)|^p+ Ce^{-pt}|x_i|^p+C\int_0^te^{-(t-s)}\mathbb{E}|Q_i(s)|^pds,
    \]
    where $C>0$. The first term on the right hand side is integrable by supposition and the third is integrable as $\mathbb{E}|Q_i(t)|^p$ has been shown to be integrable. Hence,  $\mathbb{E}|Y_i(t)|^p \in L^1(\mathbb{R}_+;\mathbb{R})$ and so $\mathbb{E}\|Y(\cdot)\|^p \in L^1(\mathbb{R}_+;\mathbb{R})$.
    
    Next we show the reverse implication. Suppose $\mathbb{E}\|Y(\cdot)\|^p \in L^1(\mathbb{R}_+;\mathbb{R})$. Introduce the process $Z(t)=X(t)-Y(t)$. For each outcome $\omega$, $Z$ satisfies the equation
    \[
    \dot{Z}(t)= \int_{[0,t]}\nu(ds)Z(t-s)+g(t),
    \]
    where $g(t)=Y(t)+\int_{[0,t]}\nu(ds)Y(t-s)$. This equation is interpreted as in the deterministic literature \cite{GLS} with analysis being carried out path by path. The $i^{th}$ component of $Z$ satisfies
    \[
    Z_i(t) =\sum_{j=1}^d r_{ij}(t)x_j + \sum_{j=1}^d \int_0^tr_{ij}(t-s)g_j(s)ds.
    \]
    This yields the estimate
    \[
    \mathbb{E}|Z_i(t)|^p \leq C\sum_{j=1}^d |r_{ij}(t)x_j|^p+ C\sum_{j=1}^d \int_0^t|r_{ij}(t-s)|\mathbb{E}|g_j(s)|^pds,
    \]
    where $C>0$. Assumption \eqref{assum. continuous resolvent is L1} ensures $r \in L^p(\mathbb{R}_+;\mathbb{R}^{d \times d})$. Thus, we need only show $\mathbb{E}|g_j(t)|^p$ is integrable. By the same argument used to estimate $\mathbb{E}|Q_i(t)|^p$, 
    \[
    \mathbb{E}|g_i(t)|^p \leq C\mathbb{E}|Y_i(t)|^p +C\sum_{j=1}^d \left(|\nu_{ij}| \ast \mathbb{E}|Y_j(\cdot)|^p\right)(t),
    \]
    where $C>0$. Once again as $|\nu_{ij}|$ is a finite measure and $\mathbb{E}\|Y(\cdot)\|^p$ is integrable (by supposition), it follows that $\mathbb{E}|g_i(t)|^p$ is integrable and hence $\mathbb{E}|Z_i(t)|^p$ is integrable. Thus, $\mathbb{E}\|Z(\cdot)\|^p \in L^1(\mathbb{R}_+;\mathbb{R})$. By the definition of $Z$, $X(t)=Z(t)+Y(t)$. Thus, each component satisfies
    \[
    \int_0^\infty\mathbb{E}|X_i(t)|^pdt \leq C\int_0^\infty\mathbb{E}|Z_i(t)|^p+\mathbb{E}|Y_i(t)|^pdt < \infty,
    \]
    where $C>0$. Thus, $\mathbb{E}\|X(\cdot)\|^p \in L^1(\mathbb{R}_+;\mathbb{R})$ as required. The second assertion follows by using the exact same estimates modulo taking expectations. 
\end{proof}
With Lemma \ref{lem. L^p of SVE and OU are equivalent} at hand we can now devote our attention to the study of the OU type process \eqref{eq. continuous OU process}. The proof of the following theorem is the main piece of analysis throughout this work.

\begin{theorem} \label{thm. L^p theorem for continuous OU}
    Let $p\in [1,\infty)$, $r$ obey \eqref{assum. continuous resolvent is L1}, and $Y$ be the solution of  \eqref{eq. continuous OU process}. Then we have the following dichotomy:
    \begin{enumerate}
        \item[(\textbf{A})] If $p \in [2,\infty)$, the following are equivalent:
    \begin{itemize}
        \item[(i)] $\mathbb{E}\|Y(\cdot)\|^p \in L^1(\mathbb{R}_+;\mathbb{R})$,
        \item[(ii)]$Y(\cdot) \in L^p(\mathbb{R}_+;\mathbb{R}^d)$ almost surely,
        \item[(iii)]$f$ obeys \eqref{cond. f} and $\sigma$ obeys \eqref{cond. sigma p geq 2}.
    \end{itemize}

    \item[(\textbf{B})] If $p \in [1,2)$, the following are equivalent:
    \begin{itemize}
        \item[(i)] $\mathbb{E}\|Y(\cdot)\|^p \in L^1(\mathbb{R}_+;\mathbb{R})$,
        \item[(ii)]$Y(\cdot) \in L^p(\mathbb{R}_+;\mathbb{R}^d)$ almost surely,
        \item[(iii)]$f$ obeys \eqref{cond. f} and $\sigma$ obeys \eqref{cond. sigma p<2}.
    \end{itemize}
    \end{enumerate}

\end{theorem}

\begin{proof}[Proof of Theorem \ref{thm. L^p theorem for continuous OU}]
\textbf{(A)}: $(iii) \implies (i).$ The $i^{th}$ component of $Y$ satisfies
\[
Y_i(t)=\int_0^te^{-(t-s)}f_i(s)ds+\sum_{j=1}^m\int_0^te^{-(t-s)}\sigma_{ij}(s)dB_j(s).
\]
This yields the estimate
\[
\mathbb{E}|Y_i(t)|^p \leq C\left| \int_0^te^{-(t-s)}f_i(s)ds\right|^p + C\sum_{j=1}^m \mathbb{E}\left[\left|\int_0^te^{-(t-s)}\sigma_{ij}(s)dB_j(s)\right|^p\right],
\]
where $C>0$. As the authors have shown in \cite[Theorem 1]{AL:2023(AppliedMathLetters)}, $\int_0^te^{-(t-s)}f_i(s)ds \in L^p(\mathbb{R}_+;\mathbb{R})$ if and only if $\int_{\cdot}^{\cdot+\theta} f_i(s)ds \in L^p(\mathbb{R}_+;\mathbb{R})$ for all $\theta>0$. Thus, we need only focus on the second term. For each $j \in \{1,\ldots,m\}$
\[
\dfrac{\int_0^te^{-(t-s)}\sigma_{ij}(s)dB_j(s)}{\left( \int_0^te^{-2(t-s)}\sigma_{ij}^2(s)ds\right)^{1/2}} \sim \mathcal{N}(0,1).
\]
Hence
\[
\mathbb{E}\left[\left|\int_0^te^{-(t-s)}\sigma_{ij}(s)dB_j(s)\right|^p\right]=C_p\left( \int_0^te^{-2(t-s)}\sigma_{ij}^2(s)ds\right)^{p/2},
\]
where $C_p$ is a constant representing the $p$-th moment of a standard normal random variable. Then once again by \cite[Theorem 1]{AL:2023(AppliedMathLetters)}, the left hand side is integrable if and only if $\int_{\cdot}^{\cdot+\theta} \sigma_{ij}^2(s)ds \in L^{\frac{p}{2}}(\mathbb{R}_+;\mathbb{R})$ for all $\theta>0$. Thus, $\mathbb{E}|Y_i(t)|^p$ is integrable and the first claim is proven.\\
\newline
\textbf{(A)}: $(i) \implies (ii).$ An application of Fubini's Theorem yields the claim.\\
\newline
\textbf{(A)}: $(ii) \implies (iii).$ Recall, the $i^{th}$ component of $Y$ obeys
\[
dY_i(t)=(f_i(t)-Y_i(t))dt+ \sum_{j=1}^m \sigma_{ij}(t)dB_j(t).
\]
Integrating this equation over the interval $[t,t+1]$ yields
\[
Y_i(t+1)-Y_i(t)+\int_{t}^{t+1}Y_i(s)ds= \int_{t}^{t+1}f_i(s)ds+\sum_{j=1}^m\int_t^{t+1}\sigma_{ij}(s)dB_j(s).
\]
Condition $(ii)$ implies the three terms on the left hand side are elements of $L^p(\mathbb{R}_+;\mathbb{R})$ almost surely. Thus,
\[
\int_0^\infty\left| \int_{t}^{t+1}f_i(s)ds+\sum_{j=1}^m\int_t^{t+1}\sigma_{ij}(s)dB_j(s) \right|^p dt < \infty \quad a.s.
\]
We break this integral into two series
\begin{multline*}
\sum_{n=0}^\infty \int_{2n}^{2n+1}\left| \int_{t}^{t+1}f_i(s)ds+\sum_{j=1}^m\int_t^{t+1}\sigma_{ij}(s)dB_j(s) \right|^pdt+\\\sum_{n=0}^\infty \int_{2n+1}^{2n+2}\left| \int_{t}^{t+1}f_i(s)ds+\sum_{j=1}^m\int_t^{t+1}\sigma_{ij}(s)dB_j(s) \right|^pdt <\infty \quad a.s.
\end{multline*}
We estimate these series below
\begin{multline*}
\sum_{n=0}^\infty \left|\int_{2n}^{2n+1} \int_{t}^{t+1}f_i(s)dsdt+\sum_{j=1}^m\int_{2n}^{2n+1} \int_t^{t+1}\sigma_{ij}(s)dB_j(s) dt\right|^p+\\\sum_{n=0}^\infty \left|\int_{2n+1}^{2n+2} \int_{t}^{t+1}f_i(s)dsdt+\sum_{j=1}^m\int_{2n+1}^{2n+2}\int_t^{t+1}\sigma_{ij}(s)dB_j(s) dt\right|^p.
\end{multline*}
Introduce the following notation
\begin{align*}
    f_{i,1}(n) & \coloneqq \int_{2n}^{2n+1} \int_{t}^{t+1}f_i(s)dsdt; \quad f_{i,2}(n) \coloneqq \int_{2n+1}^{2n+2} \int_{t}^{t+1}f_i(s)dsdt,\\
    u_{i,j}(n) & \coloneqq \int_{2n}^{2n+1} \int_t^{t+1}\sigma_{ij}(s)dB_j(s) dt; \quad v_{i,j}(n) \coloneqq \int_{2n+1}^{2n+2}\int_t^{t+1}\sigma_{ij}(s)dB_j(s) dt.
\end{align*}
The two series become
\[
\sum_{n=0}^\infty \left| f_{i,1}(n)+\sum_{j=1}^mu_{i,j}(n)\right|^p+\sum_{n=0}^\infty \left| f_{i,2}(n)+\sum_{j=1}^mv_{i,j}(n)\right|^p < \infty \quad a.s.
\]
An application of the stochastic Fubini Theorem yields the following representations for $u_{i,j}(n)$ and $v_{i,j}(n)$ respectively,
\begin{align*}
 u_{i,j}(n) & = \int_{2n}^{2n+2}\left[ s \wedge (2n+1)-2n\vee (s-1) \right] \sigma_{ij}(s)dB_j(s),\\
 v_{i,j}(n) & = \int_{2n+1}^{2n+3}\left[s\wedge(2n+2)-(2n+1)\vee(s-1)\right]\sigma_{ij}(s)dB_j(s).
\end{align*}
We note $u_{i,j}(n)$ and $v_{i,j}(n)$ are sequences of independent Gaussian random variables (with mean zero), this is because each member of the sequences is a Wiener integral where the domains of integration are non-overlapping. It is a general fact that such integrals are independent. Furthermore, we define
\[
U_i(n) \coloneqq \sum_{j=1}^m u_{i,j}(n); \quad V_i(n) \coloneqq \sum_{j=1}^m v_{i,j}(n).
\]
$U_i(n)$ and $V_i(n)$ are sequences of independent Gaussian random variables (which follows from the independence of the components of the $m$-dimensional Brownian motion). These can be rewritten as follows\footnote{As in the discrete case we need not worry that $\mathbb{E}|U_i(n)|^2=0$ or $\mathbb{E}|V_i(n)|^2=0$ for some $n$, in this case we define $\tilde{U}_i(n)\coloneqq U_i(n)$ and $\tilde{V}_i(n)=V_i(n)$ for such values of $n$. If however both $\mathbb{E}|U_i(n)|^2=0$ and $\mathbb{E}|V_i(n)|^2=0$ for all $n$ then care is needed. Using the identity \eqref{eq. Variance of U_i and V_i}, It\^o's isometry will force $\sigma_{ij}(t)=0$ for $t \in [2n,2n+2]$. But this holds for all $n$ and so it forces $\sigma_{ij}(t)=0$ for all $t\geq0$ and so the desired integrability condition will be trivially satisfied.}
\begin{align*}
    U_i(n) & = \sqrt{\mathbb{E}|U_i(n)|^2} \dfrac{U_i(n)}{\sqrt{\mathbb{E}|U_i(n)|^2}} = \sqrt{\mathbb{E}|U_i(n)|^2} \tilde{U}_i(n),\\
    V_i(n) & = \sqrt{\mathbb{E}|V_i(n)|^2} \dfrac{V_i(n)}{\sqrt{\mathbb{E}|V_i(n)|^2}} = \sqrt{\mathbb{E}|V_i(n)|^2} \tilde{V}_i(n).\\
\end{align*}
Now $\tilde{U}_i(n)$ and $\tilde{V}_i(n)$ are i.i.d. sequences of standard normal random variables. This yields the final form of our two series
\begin{equation} \label{eq. Variance of U_i and V_i}
  \sum_{n=0}^\infty \left| f_{i,1}(n)+\sqrt{\mathbb{E}|U_i(n)|^2} \tilde{U}_i(n)\right|^p+\sum_{n=0}^\infty \left| f_{i,2}(n)+\sqrt{\mathbb{E}|V_i(n)|^2} \tilde{V}_i(n)\right|^p < \infty \quad a.s.  
\end{equation}
Invoking Lemma \ref{lem. summability converse} yields
\[
\sum_{n=0}^\infty \left|\sqrt{\mathbb{E}|U_i(n)|^2}\right|^p+\sum_{n=0}^\infty \left|\sqrt{\mathbb{E}|V_i(n)|^2}\right|^p < \infty.
\]
Next we analyse both summands. As $U_i$ and $V_i$ are linear combinations of independent normal random variables (with zero mean) it follows

\begin{align*}
    \mathbb{E}|U_i(n)|^2 =\sum_{j=1}^m\mathbb{E}|u_{i,j}(n)|^2; \quad \mathbb{E}|V_i(n)|^2 =\sum_{j=1}^m\mathbb{E}|v_{i,j}(n)|^2
\end{align*}
As norms on finite dimensional vector spaces are equivalent, we have
\begin{align*}
    \sum_{n=0}^\infty \left|\sqrt{\mathbb{E}|U_i(n)|^2}\right|^p & =\sum_{n=0}^\infty \left(\sum_{j=1}^m\mathbb{E}|u_{i,j}(n)|^2\right)^{\frac{p}{2}} \geq C_1\sum_{n=0}^\infty \left(\sum_{j=1}^m \sqrt{\mathbb{E}|u_{i,j}(n)|^2}\right)^p,\\
    \sum_{n=0}^\infty \left|\sqrt{\mathbb{E}|V_i(n)|^2}\right|^p & =\sum_{n=0}^\infty \left(\sum_{j=1}^m\mathbb{E}|v_{i,j}(n)|^2\right)^{\frac{p}{2}} \geq C_2\sum_{n=0}^\infty \left(\sum_{j=1}^m \sqrt{\mathbb{E}|v_{i,j}(n)|^2}\right)^p,
\end{align*}
with $C_1,C_2>0$. Furthermore as $p\geq1$ and both $\mathbb{E}|u_{i,j}(n)|^2$ and $\mathbb{E}|v_{i,j}(n)|^2$ are non-negative,
\begin{align*}
    \infty > \sum_{n=0}^\infty \left|\sqrt{\mathbb{E}|U_i(n)|^2}\right|^p & \geq C_1\sum_{n=0}^\infty \left(\sum_{j=1}^m \sqrt{\mathbb{E}|u_{i,j}(n)|^2}\right)^p \geq C_1\sum_{n=0}^\infty \sum_{j=1}^m \left(\mathbb{E}|u_{i,j}(n)|^2\right)^{\frac{p}{2}},\\
    \infty > \sum_{n=0}^\infty \left|\sqrt{\mathbb{E}|V_i(n)|^2}\right|^p & \geq C_2\sum_{n=0}^\infty \left(\sum_{j=1}^m \sqrt{\mathbb{E}|v_{i,j}(n)|^2}\right)^p \geq C_2\sum_{n=0}^\infty \sum_{j=1}^m \left(\mathbb{E}|v_{i,j}(n)|^2\right)^{\frac{p}{2}}.
\end{align*}
Thus it must be the case that
\[
\sum_{n=0}^\infty \left(\mathbb{E}|u_{i,j}(n)|^2\right)^{\frac{p}{2}} + \sum_{n=0}^\infty \left(\mathbb{E}|v_{i,j}(n)|^2\right)^{\frac{p}{2}} < \infty,
\]
for each $j \in \{1,\ldots,m\}$. Employing It\^o's isometry for each $j$
\begin{align*}
    \sum_{n=0}^\infty \left(\mathbb{E}|u_{i,j}(n)|^2\right)^{\frac{p}{2}}& = \sum_{n=0}^\infty \left(\int_{2n}^{2n+2}\left[ s \wedge (2n+1)-2n\vee (s-1) \right]^2 \sigma_{ij}^2(s)ds\right)^{\frac{p}{2}} < \infty,\\
    \sum_{n=0}^\infty \left(\mathbb{E}|v_{i,j}(n)|^2\right)^{\frac{p}{2}}& = \sum_{n=0}^\infty \left(\int_{2n+1}^{2n+3}\left[s\wedge(2n+2)-(2n+1)\vee(s-1)\right]^2\sigma_{ij}^2(s)ds\right)^{\frac{p}{2}} < \infty.
\end{align*}
Next we obtain a lower estimate for the summands for both series.
\begin{align*}
    \mathbb{E}|u_{i,j}(n)|^2 & = \int_{2n}^{2n+2}\left[ s \wedge (2n+1)-2n\vee (s-1) \right]^2 \sigma_{ij}^2(s)ds\\
    & = \int_{2n}^{2n+1}(s-2n)^2\sigma_{ij}^2(s)ds+\int_{2n+1}^{2n+2}(2n+2-s)^2\sigma_{ij}^2(s)ds\\
    & \geq \int_{2n+\frac{1}{2}}^{2n+1}\frac{\sigma_{ij}^2(s)}{4}ds+\int_{2n+1}^{2n+\frac{3}{2}}\frac{\sigma_{ij}^2(s)}{4}ds.
\end{align*}
Similarly
\begin{align*}
    \mathbb{E}|v_{i,j}(n)|^2& = \int_{2n+1}^{2n+3}\left[s\wedge(2n+2)-(2n+1)\vee(s-1)\right]^2\sigma_{ij}^2(s)ds\\
    & = \int_{2n+1}^{2n+2}(s-2n-1)^2\sigma_{ij}^2(s)ds+\int_{2n+2}^{2n+3}(2n+3-s)^2\sigma_{ij}^2(s)ds\\
    & \geq \int_{2n+\frac{3}{2}}^{2n+2}\frac{\sigma_{ij}^2(s)}{4}ds+\int_{2n+2}^{2n+\frac{5}{2}}\frac{\sigma_{ij}^2(s)}{4}ds.
\end{align*}
Now consider
\[
\frac{1}{C}\left|\int_{2n}^{2n+2}\sigma_{ij}^2(s)ds\right|^{p/2} \leq \left|\int_{2n}^{2n+\frac{1}{2}}\sigma_{ij}^2(s)ds\right|^{p/2}+\left|\int_{2n+\frac{1}{2}}^{2n+\frac{3}{2}}\sigma_{ij}^2(s)ds\right|^{p/2}+\left|\int_{2n+\frac{3}{2}}^{2n+2}\sigma_{ij}^2(s)ds\right|^{p/2},
\]
where $C>0$. The last two terms on the right hand side are summable by the above estimates on $\mathbb{E}|u_{i,j}(n)|^2$ and $\mathbb{E}|v_{i,j}(n)|^2$. To see that the first term is summable consider
\[
\sum_{n=0}^\infty \left| \int_{2n+2}^{2n+\frac{5}{2}}\sigma_{ij}^2(s)ds \right|^{p/2}.
\]
We know this sum is finite by virtue of the above estimates on $\mathbb{E}|v_{i,j}(n)|^2$. Thus, making the change of indexing variable $n=l-1$ yields
\[
\sum_{l=1}^\infty \left| \int_{2l}^{2l+\frac{1}{2}}\sigma_{ij}^2(s)ds \right|^{p/2}.
\]
Hence
\begin{equation} \tag{$\dagger$}
  \sum_{n=0}^\infty \left|\int_{2n}^{2n+2}\sigma_{ij}^2(s)ds\right|^{p/2} < \infty. 
\end{equation}
Finally,
\begin{align*}
    \int_0^\infty\left|\int_t^{t+1} \sigma_{ij}^2(s)ds\right|^{p/2}dt & = \sum_{n=0}^{\infty}\int_{2n}^{2n+2}\left|\int_t^{t+1} \sigma_{ij}^2(s)ds\right|^{p/2}dt\\
    & \leq \sum_{n=0}^{\infty}\int_{2n}^{2n+2}\left|\int_{2n}^{2n+3} \sigma_{ij}^2(s)ds\right|^{p/2}dt \\
    & = 2 \sum_{n=0}^{\infty}\left|\int_{2n}^{2n+3} \sigma_{ij}^2(s)ds\right|^{p/2}\\
    & \leq C\sum_{n=0}^{\infty}\left|\int_{2n}^{2n+2} \sigma_{ij}^2(s)ds\right|^{p/2} + C\sum_{n=0}^{\infty}\left|\int_{2n+2}^{2n+3} \sigma_{ij}^2(s)ds\right|^{p/2},
\end{align*}
where $C>0$. We have proven the first term is finite which also implies the second is finite after a change of indexing variable. Thus,
\[
\int_\cdot^{\cdot+1} \sigma_{ij}^2(s)ds \in L^{\frac{p}{2}}(\mathbb{R}_+;\mathbb{R}).
\]
We note this is equivalent to $\int_\cdot^{\cdot+\theta} \sigma_{ij}^2(s)ds \in L^{\frac{p}{2}}(\mathbb{R}_+;\mathbb{R})$ for all $\theta>0$. As $i$ and $j$ were arbitrary, this holds for all components of the matrix $\sigma$. Recall
\[
\int_\cdot^{\cdot+\theta} \sigma_{ij}^2(s)ds \in L^{\frac{p}{2}}(\mathbb{R}_+;\mathbb{R}) \text{ for all } \theta>0   \iff  \mathbb{E}\left[\left|\int_0^\cdot e^{-(\cdot-s)}\sigma_{ij}(s)dB_j(s)\right|^p\right] \in L^1(\mathbb{R}_+;\mathbb{R}).
\]
Thus,
\begin{equation} \tag{$\dagger'$}
 \int_0^\infty \left|\int_0^t e^{-(t-s)}\sigma_{ij}(s)dB_j(s)\right|^p dt < \infty \quad a.s.   
\end{equation}
Recall
\[
\int_0^te^{-(t-s)}f_i(s)ds=Y_i(t)-\sum_{j=1}^m \int_0^t e^{-(t-s)}\sigma_{ij}(s)dB_j(s).
\]
As all terms on the right are elements of $L^p(\mathbb{R}_+;\mathbb{R})$ almost surely, it follows the left hand side is also an element of $L^p(\mathbb{R}_+;\mathbb{R})$. But once again by \cite[Theorem 1]{AL:2023(AppliedMathLetters)} this is equivalent to
\[
\int_{\cdot}^{\cdot+\theta} f_i(s)ds \in L^p(\mathbb{R}_+;\mathbb{R}) \text{ for all } \theta >0.
\]
As this holds for all $i \in \{1,\ldots,d\}$, the proof of \textbf{(A)} is complete.\\
\newline
\textbf{(B)}: $(iii) \implies (i).$ Follow the proof from \textbf{(A)} and replace the $L^p$ condition on $\int_t^{t+\theta}\sigma_{ij}^2(s)ds$ with the summability condition $\int_n^{n+1}\sigma_{ij}^2(s)ds \in \ell^{\frac{p}{2}}(\mathbb{N};\mathbb{R})$ and invoke Lemma \ref{lem. pertured ODE in Lp p<1}.\\
\newline
\textbf{(B)}: $(i) \implies (ii).$ Follow the proof from \textbf{(A)}.\\
\newline
\textbf{(B)}: $(ii) \implies (iii).$ Follow the proof from \textbf{(A)} up until $(\dagger)$. Thus,
\[
\sum_{n=0}^\infty \left|\int_n^{n+1}\sigma_{ij}^2(s)ds\right|^\frac{p}{2} < + \infty.
\]
By Lemma \ref{lem. pertured ODE in Lp p<1}
\[
\int_\cdot^{\cdot+1} \sigma_{ij}^2(s)ds \in \ell^{\frac{p}{2}}(\mathbb{N};\mathbb{R})   \iff  \mathbb{E}\left[\left|\int_0^\cdot e^{-(\cdot-s)}\sigma_{ij}(s)dB_j(s)\right|^p\right] \in L^1(\mathbb{R}_+;\mathbb{R}).
\]
Then follow the proof of \textbf{(A)} from $(\dagger')$ onwards.

\end{proof}

\begin{lemma} \label{lem. pertured ODE in Lp p<1}
Let $f$ be a non-negative continuous function and $\beta >0$. Then for $p\in (0,1)$, the following are equivalent:
\begin{itemize}
    \item[(i)] $t \mapsto \int_0^te^{-\beta(t-s)}f(s)ds \in L^p(\mathbb{R}_+;\mathbb{R})$,
    \item[(ii)]$\int_n^{n+1}f(s)ds \in \ell^p(\mathbb{N};\mathbb{R})$. 
\end{itemize}
\end{lemma}
\begin{proof}[Proof of Lemma \ref{lem. pertured ODE in Lp p<1}]
    $(ii) \implies (i)$: Define $v(t)\coloneqq \int_0^te^{-\beta(t-s)}f(s)ds $. From the definition of $v$, we have for all $t \in [n,n+1]$
    \begin{equation} \label{eq. perturbed ODE alternate representation}
        v(t)=v(n)e^{-\beta(t-n)}+\int_n^te^{-\beta(t-s)}f(s)ds. 
    \end{equation}
    Noting that $v$ is non-negative yields
    \begin{align*}
        v(t) \leq v(n)+\int_n^{n+1}f(s)ds.
    \end{align*}
    Raising both sides to the power of $p$ and integrating
    \[
    \int_n^{n+1}v^{p}(t)dt \leq \int_n^{n+1}\left(v(n)+\int_n^{n+1}f(s)ds\right)^pdt=\left(v(n)+\int_n^{n+1}f(s)ds\right)^p.
    \]
     Summing over $n$ yields
    \[
    \int_0^\infty v^p(t)dt \leq \sum_{n=0}^\infty v^p(n) + \sum_{n=0}^\infty \left(\int_n^{n+1}f(s)ds\right)^p.
    \]
    The second series is finite by assumption. Thus, we need only show $v(n) \in \ell^p(\mathbb{N};\mathbb{R})$. Setting $t=n+1$ in \eqref{eq. perturbed ODE alternate representation} and estimating yields
    \[
    v(n+1) \leq v(n)e^{-\beta}+\int_{n}^{n+1}f(s)ds.
    \]
    Thus,
    \[
    v(n) \leq \sum_{j=1}^ne^{-\beta(n-j)}\left(\int_{j-1}^jf(s)ds\right).
    \]
    Hence
    \[
    v^p(n) \leq \sum_{j=1}^ne^{-\beta p(n-j)}\left(\int_{j-1}^jf(s)ds\right)^p.
    \]
    The right hand side is a convolution of two summable sequences and is thus summable. Hence $v^p(n) \in \ell^1(\mathbb{N};\mathbb{R})$ and we are done.\\

    $(i) \implies (ii)$: Returning once again to \eqref{eq. perturbed ODE alternate representation}, we get the estimate
    \[
    v(t) \geq v(n)e^{-\beta(t-n)} \geq v(n)e^{-\beta}.
    \]
    Thus,
    \[
    \int_n^{n+1}v^p(t)dt \geq v^p(n)e^{-\beta p}.
    \]
    Summing we obtain
    \[
    \infty > \int_0^\infty v^p(t)dt \geq \sum_{n=0}^\infty v^p(n)e^{-\beta p}.
    \]
    Hence $v(n) \in \ell^p(\mathbb{N};\mathbb{R})$. But once again with $t=n+1$ in \eqref{eq. perturbed ODE alternate representation} we obtain
    \begin{align*}
    v(n+1)& =v(n)e^{-\beta}+e^{-\beta(n+1)}\int_n^{n+1}e^{\beta s}f(s)ds\\
    & \geq e^{-\beta}\int_n^{n+1}f(s)ds.
    \end{align*}
    Thus, $\int_n^{n+1}f(s)ds \in \ell^p(\mathbb{N};\mathbb{R})$. The claim is proven.

\end{proof}
The proof of Theorem \ref{thm. L^p theorem continuous SVE} now follows.

\begin{proof}[Proof of Theorem \ref{thm. L^p theorem continuous SVE}]
Theorem \ref{thm. L^p theorem for continuous OU} along with Lemma \ref{lem. L^p of SVE and OU are equivalent} yield the claim.
    
\end{proof}
\subsection{Almost sure asymptotic behaviour} \label{subsec. Almost sure asymptotic behaviour}
When within the regime of $p-$integrable sample paths of solutions to \eqref{eq. continuous SVE}, we show that the asymptotic behaviour of the sample paths is determined by the asymptotic behaviour of the underlying deterministic equation. Throughout this section we use the shorthand notation
\begin{equation} \label{eq. convolution notation}
    (f \ast g)(t)\coloneqq\int_0^tf(t-s)g(s)ds
\end{equation}
to denote the convolution of two functions $f$ and $g$ where the product in the integrand is regular matrix multiplication. If $f$ takes values in $\mathbb{R}^{d \times m}$ and $g$ takes values in $\mathbb{R}^{m \times n}$, then the convolution $f \ast g$ is the $\mathbb{R}^{d \times n}$-valued function with components
\begin{align*}
    (f \ast g)(t)_{i,j}=\int_0^t\langle \textbf{e}_i^\top f(t-s),g(s)\textbf{e}_j\rangle ds.
\end{align*}

\begin{theorem} \label{thm. pathwise behaviour of X}
    Let $p\in[1,\infty)$, $X$ be the solution of \eqref{eq. continuous SVE} and $r \in L^1(\mathbb{R}_+;\mathbb{R}^{d\times d}  )$ be the solution of \eqref{eq. continuous resolvent}. Assume $X \in L^p(\mathbb{R}_+;\mathbb{R}^d)$ almost surely. Then
    \begin{equation} \label{eq. X-r to 0}
        \left\|X(t)-(r \ast f)(t)\right\| \longrightarrow 0 \text{ as } t \to \infty \quad \text{ a.s}. 
    \end{equation}
\end{theorem}
Theorem \ref{thm. pathwise behaviour of X} precisely identifies the almost sure asymptotic behaviour of the sample paths in terms of the resolvent $r$, which could be seen as unsatisfactory. In general within this regime, the asymptotic behaviour of the mapping $t \mapsto (r \ast f)(t)$ cannot easily be determined. If one wishes for solutions to vanish almost surely, it is necessary that $(r \ast f)(t) \to 0$ as $t \to \infty$ which as outlined in \cite{AL:2023(AppliedMathLetters)}, is equivalent to $t \mapsto \int_t^{t+\theta}f(s)ds \to 0$ as $t \to \infty$ for all $\theta >0$. However, it is not evident that this will be the case. In the special case when $X \in L^1(\mathbb{R}_+;\mathbb{R}^d)$ almost surely, one can remove the dependence on the resolvent in \eqref{eq. X-r to 0} and obtain the almost sure asymptotic behaviour of the trajectories purely in terms of the perturbation function $f$.
\begin{corollary} \label{cor. pathwise behaviour of X p=1}
    Let $X$ be the solution of \eqref{eq. continuous SVE} and $r \in L^1(\mathbb{R}_+;\mathbb{R}^{d\times d}  )$ be the solution of \eqref{eq. continuous resolvent}. Assume $X \in L^1(\mathbb{R}_+;\mathbb{R}^d)$ $a.s.$ Then
    \begin{equation}
        \left\|X(t)-\int_0^1\int_{t-u}^tf(s)dsdu\right\| \longrightarrow 0 \text{ as } t \to \infty \quad \text{a.s.},
    \end{equation}
    where $\int_0^1\int_{t-u}^tf(s)dsdu$ is a $d$-dimensional vector whose $i^{th}$ entry is given by $\int_0^1\int_{t-u}^tf_i(s)dsdu$.
    
\end{corollary}
\begin{proof}[Proof of Corollary \ref{cor. pathwise behaviour of X p=1}]
    By Theorem \ref{thm. pathwise behaviour of X}
    \begin{equation}
        \left\|X(t)-(r \ast f)(t)\right\| \longrightarrow 0 \text{ as } t \to \infty \quad \text{a.s}. 
    \end{equation}
    Hence we need only focus on the asymptotic behaviour of $(r \ast f)(t)$. Theorem \ref{thm. L^p theorem continuous SVE} implies  $\int_{\cdot}^{\cdot+\theta} f_i(s)ds \in L^1(\mathbb{R}_+;\mathbb{R})$ for all $\theta>0$, $i \in \{1,\ldots,d\}$. Hence we invoke Lemma 2 in \cite{AL:2023(AppliedMathLetters)} to claim $f_i=f_{i,1}+f_{i,2}$ where $f_{i,1} \in L^1(\mathbb{R}_+;\mathbb{R})$ and $\int_0^tf_{i,2}(s)ds \in L^1(\mathbb{R}_+;\mathbb{R})$. We write the vector-valued function $f=f_1+f_2$ where $f_1$ and $f_2$ are defined in the obvious way. Now with $f_3(t)\coloneqq \int_0^tf_2(s)ds$ we follow the proof of Theorem 1 from \cite{AL:2023(AppliedMathLetters)} to obtain
    \begin{align} \label{eq. r ast f}
    (r\ast f)(t)& =(r\ast f_1)(t)+f_3(t)+(r' \ast f_3)(t) \nonumber \\ 
    & = (r\ast f_1)(t)+f_3(t)+(\nu \ast r \ast f_3)(t),
    \end{align}
    where the last equality follows from equation \eqref{eq. continuous resolvent}. As $r \in L^1(\mathbb{R}_+;\mathbb{R}^{d \times d})$ we know also that $r \in BC_0(\mathbb{R}_+;\mathbb{R}^{d \times d})$ and so $r \ast f_1 \in BC_0(\mathbb{R}_+;\mathbb{R}^{d})$. Additionally as $\nu \in M(\mathbb{R}_+;\mathbb{R}^{d \times d} )$ and $\nu \ast r \in BC_0(\mathbb{R}_+;\mathbb{R}^{d \times d})$ we have $\nu \ast r \ast f_3 \in BC_0(\mathbb{R}_+;\mathbb{R}^{d})$. Thus,
    \[
    \|X(t)-f_3(t)\| \leq \left\|X(t)-(r \ast f)(t)\right\|+\| (r\ast f_1)(t) \| + \| (\nu \ast r \ast f_3)(t) \|.
    \]
    Hence $\|X(t)-f_3(t)\| \to 0$ as $t \to \infty$ almost surely. The explicit representation of $f_3$ is given in the proof of Lemma 2 in \cite{AL:2023(AppliedMathLetters)}. Namely,
    \[
    f_3(t)= \int_0^1\int_{t-u}^tf(s)dsdu.
    \]
    The claim is proven.
\end{proof}
In the introduction we discussed that often times a strategy to infer information about asymptotic behaviour of solutions is to first prove some kind of integrability result and then infer convergence to zero. We will use Theorem \ref{thm. pathwise behaviour of X} and Corollary \ref{cor. pathwise behaviour of X p=1} to highlight how this line of attack may yield severely suboptimal results for even the simplest equations. We shall momentarily kill the deterministic perturbation $f$ and consider only the scalar equation
\begin{equation} \label{eq. scalar version of Y}
    dY(t)=-Y(t)dt+\sigma(t)dB(t); \quad Y(0)=0.
\end{equation}
Now $\sigma \in C(\mathbb{R}_+;\mathbb{R})$ and $B$ is a scalar Brownian motion. Theorem \ref{thm. pathwise behaviour of X} ensures the condition $n \mapsto \int_n^{n+1}\sigma^2(s)ds \in \ell^{1/2}(\mathbb{N};\mathbb{R})$ is enough to force $|Y(t)| \to 0$ as $t \to \infty$ almost surely. However, it was shown in \cite{ACR:2011(DCDS)} that $Y$ converging to zero almost surely is equivalent to
\begin{equation} \label{eq. S(e)}
    \sum_{n=0}^\infty \sqrt{\int_n^{n+1}\sigma^2(s)ds}\exp\left(\frac{-\varepsilon}{\int_n^{n+1}\sigma^2(s)ds}\right)< \infty,
\end{equation}
for all $\varepsilon >0$. But
\[
 \sum_{n=0}^\infty \sqrt{\int_n^{n+1}\sigma^2(s)ds}\exp\left(\frac{-\varepsilon}{\int_n^{n+1}\sigma^2(s)ds}\right)\leq \sum_{n=0}^\infty \sqrt{\int_n^{n+1}\sigma^2(s)ds},
\]
uniformly in $\varepsilon$. Hence if one thinks about convergence to zero in terms of condition \eqref{eq. S(e)}, $Y$ being almost surely integrable is most definitely sufficient but nowhere near necessary. Reintroducing the deterministic perturbation $f$ makes the situation even worse. Theorem \ref{thm. pathwise behaviour of X} and Corollary \ref{cor. pathwise behaviour of X p=1} demonstrate integrability is no longer sufficient, and in fact one needs to impose side conditions on $f$ to ensure convergence to zero. Next we provide a convergence characterisation in the special case when $\sigma$ is a diagonal matrix.

\begin{theorem} \label{thm. characterising convergence of X a.s. diagonal sigma}
Let $r$ obey \eqref{assum. continuous resolvent is L1}, $X$ be the solution of  \eqref{eq. continuous SVE}, $m=d$, and $\sigma$ be a diagonal matrix. Then the following are equivalent:
\begin{enumerate}
    \item[(i)] Each component of $f$ and $\sigma$ satisfies
    \begin{align} \label{eq. conditions on f sigma for asymptotic convergence}
        \int_t^{t+\theta}f_i(s)ds \overset{t \to \infty}{\longrightarrow} 0 \text{ for all } \theta >0; \quad    \sum_{n=0}^\infty \sqrt{\int_n^{n+1}\sigma_{ii}^2(s)ds}\exp\left(\frac{-\varepsilon}{\int_n^{n+1}\sigma_{ii}^2(s)ds}\right)< \infty
    \end{align}
    for all $\varepsilon >0$,
    \item[(ii)] $\|X(t)\| \longrightarrow 0 $ as $t \to \infty$, almost surely.
\end{enumerate}
\end{theorem}
\begin{proof}[Proof of Theorem \ref{thm. characterising convergence of X a.s. diagonal sigma}]
The process satisfying $dY(t)=(f(t)-Y(t))dt+\sigma(t)dB(t)$ is now just a concatenation of scalar processes in which almost sure convergence to zero for each component is equivalent to $(i)$ \cite[Theorem 11]{AL:2024(Pantograph)}. As in the proof of Lemma \ref{lem. L^p of SVE and OU are equivalent}, consider the process $Z(t)=X(t)-Y(t)$ and use the fact $Y \in BC_0(\mathbb{R}_+;\mathbb{R}^d)\, a.s.$  and $\nu \in M(\mathbb{R}_+;\mathbb{R}^{d \times d})$ to conclude $Z \in BC_0(\mathbb{R}_+;\mathbb{R}^d)$. As $X(t)=Z(t)+Y(t)$ we have $X \in BC_0(\mathbb{R}_+;\mathbb{R}^d)\, a.s.$  and the forward implication is proven. Now we need only show $X \in BC_0(\mathbb{R}_+;\mathbb{R}^d) \quad a.s. $ implies $ Y \in BC_0(\mathbb{R}_+;\mathbb{R}^d) \quad a.s.$, after which we can conclude by applying Theorem 11 in \cite{AL:2024(Pantograph)}. We once again continue as in the proof of Lemma \ref{lem. L^p of SVE and OU are equivalent} and consider     
\[
  dX(t)=\left(f(t)+Q(t)-X(t)\right)dt+\sigma(t)dB(t),
\]
where $Q(t)=X(t)+\int_{[0,t]}\nu(ds)X(t-s)$. As $\nu \in M(\mathbb{R}_+;\mathbb{R}^{d \times d})$ and $X \in BC_0(\mathbb{R}_+;\mathbb{R}^d)$ almost surely, we must have $ Q \in BC_0(\mathbb{R}_+;\mathbb{R}^d)$ almost surely. We obtain for each component
\[
Y_i(t)=X_i(t)-x_i e^{-t}-\int_0^te^{-(t-s)}Q_i(s)ds
\]
from which it is clear $Y_i \in BC_0(\mathbb{R}_+;\mathbb{R})$ almost surely for each component. The theorem is proven.
    
\end{proof}
We conclude this section with a proof of Theorem \ref{thm. pathwise behaviour of X}. We need the following lemma which is proven in the appendix.
\begin{lemma} \label{lem. integrability condition implies summability condition}
Let $f\in L^1_{loc}(\mathbb{R}_+;\mathbb{R}_+)$ and $p \in[1,\infty)$. Assume

\[
\int_0^\infty \left(\int_t^{t+\theta}f(s)ds\right)^p dt < \infty \quad \text{ for all } \theta >0. 
\]
Then for each sequence $(a_n)_{n\geq0}$ with $a_0=0$ and $0<\alpha\leq a_{n+1}-a_n\leq \beta$ where $\alpha,\beta>0
$, we have

\[
\sum_{n=0}^\infty \left(\int_{a_n}^{a_{n+1}}f(s)ds\right)^p < \infty.
\]
    
\end{lemma}

\begin{proof}[Proof of Theorem \ref{thm. pathwise behaviour of X}]  Introduce the $\mathbb{R}^d$-valued process with dynamics $dM(t)=-M(t)dt + \sigma(t)dB(t)$ with zero initial condition. By Theorem \ref{thm. L^p theorem continuous SVE} and Lemma \ref{lem. L^p of SVE and OU are equivalent}, $X \in L^p(\mathbb{R}_+;\mathbb{R}^d)$ forces $M \in L^p(\mathbb{R}_+;\mathbb{R}^d)$ and in particular $\mathbb{E}\left[\left|M(t)\right|^p\right] \in L^1(\mathbb{R}_+;\mathbb{R}^d)$. As $t \mapsto \mathbb{E}\left[\left|M_i(t)\right|^p\right]$ is continuous and integrable, for each $i \in \{1,\cdots, d\}$ there exists a sequence $a_n \nearrow \infty $ with $0\leq a_{n+1}-a_n \leq 1$ such that $\sum_{n=0}^\infty \mathbb{E}\left[\left|M_i(a_n)\right|^p\right] < \infty$. To see this note,

    \begin{align*}
        \infty > \int_0^\infty \mathbb{E}&\left[\left|M_j(s)\right|^p\right]ds  \\
        & = \sum_{n=0}^\infty \int_{n}^{n+1} \mathbb{E}\left[\left|M_j(s)\right|^p\right]ds \geq \sum_{n=0}^\infty \inf_{s\in [n,n+1]} \mathbb{E}\left[\left|M_j(s)\right|^p\right] = \sum_{n=0}^\infty \mathbb{E}\left[\left|M_j(a_n)\right|^p\right],
    \end{align*}
    where for each $n \in \mathbb{N}$,
    \[
    a_n \coloneqq \min\left\{ x \in [n,n+1] : \mathbb{E}\left[\left|M_j(x)\right|^p\right] = \inf_{s\in [n,n+1]} \mathbb{E}\left[\left|M_j(s)\right|^p\right]\right\}.
    \]
    The continuity of $t \mapsto \mathbb{E}\left[\left|M_j(t)\right|^p\right]$ ensures the infimum over compact intervals is always attained and by construction $0\leq a_{n+1}-a_n \leq 1$ and $a_n \nearrow \infty $. For $t \in [a_n,a_{n+1}]$ each component satisfies
    \[
    M_i(t)e^{t}=e^{t}M_i(a_n)+ \sum_{j=1}^m\int_{a_n}^te^s\sigma_{ij}(s)dB_j(s).
    \]
    Using standard estimates we obtain the following for some $C>0$,
    \begin{equation} \tag{$\ast$} \label{line before BDG inequality}
          \mathbb{E}\left[\sup_{a_n \leq t \leq a_{n+1}}\left|e^t M_i(t)\right|^p\right] \leq Ce^{pa_{n+1}}\mathbb{E}|M_i(a_n)|^p+C\sum_{j=1}^m \mathbb{E}\left[\sup_{a_n \leq t \leq a_{n+1}}\left|\int_{a_n}^te^s\sigma_{ij}(s)dB_j(s)\right|^p\right].  
    \end{equation}
    We apply the Burkholder-Davis-Gundy Inequality to each term in the finite sum on the right hand side; this yields for some $C>0$,
    \begin{align*}
        \mathbb{E}\left[\sup_{a_n \leq t \leq a_{n+1}}\left|\int_{a_n}^te^s\sigma_{ij}(s)dB_j(s)\right|^p\right] & \leq C \left(\int_{a_n}^{a_{n+1}}e^{2s}\sigma_{ij}^2(s)ds\right)^{p/2}\\
        & \leq Ce^{p(a_{n+1})}\left(\int_{a_n}^{a_{n+1}}\sigma_{ij}^2(s)ds\right)^{p/2}.
    \end{align*}
    Plugging this estimate back into \eqref{line before BDG inequality}, we obtain
    \begin{align*}
        \mathbb{E}\left[\sup_{a_n \leq t \leq a_{n+1}}\left|e^t M_i(t)\right|^p\right] \leq C'e^{p(a_{n+1})}\mathbb{E}|M_i(a_n)|^p+C'e^{p(a_{n+1})}\left(\int_{a_n}^{a_{n+1}}\sigma_{ij}^2(s)ds\right)^{p/2},
    \end{align*}
    for some new constant $C'>0$. Dividing by $e^{p(a_{n+1})}$ on both sides yields
    \begin{align*}
        e^{-p(a_{n+1}-a_n)}\mathbb{E}\left[\sup_{a_n \leq t \leq a_{n+1}}\left|M_i(t)\right|^p\right] \leq C'\mathbb{E}|M_i(a_n)|^p+C'\left(\int_{a_n}^{a_{n+1}}\sigma_{ij}^2(s)ds\right)^{p/2}.
    \end{align*}
    The first term on the right hand side is summable by construction of the sequence $a_n$. To show the second term is summable we distinguish two cases. The first case is $p \in [1,2)$. By the construction of the sequence $a_n$
    \[
    \left(\int_{a_n}^{a_{n+1}}\sigma_{ij}^2(s)ds\right)^{p/2} \leq \left(\int_{n}^{n+1}\sigma_{ij}^2(s)ds+\int_{n+1}^{n+2}\sigma_{ij}^2(s)ds\right)^{p/2}.
    \]
    Once again we use the fact that norms on finite dimensional vector spaces are equivalent to obtain the bound
    \[
    \left(\int_{n}^{n+1}\sigma_{ij}^2(s)ds+\int_{n+1}^{n+2}\sigma_{ij}^2(s)ds\right)^{p/2} \leq C^{1/p}\left(\left(\int_{n}^{n+1}\sigma_{ij}^2(s)ds\right)^{1/2}+\left(\int_{n+1}^{n+2}\sigma_{ij}^2(s)ds\right)^{1/2}\right)^{p},
    \]
    for some $C>0$. As $p\geq 1$, we use a standard estimate to obtain for some new constant $C'>0$
    \[
    \left(\int_{a_n}^{a_{n+1}}\sigma_{ij}^2(s)ds\right)^{p/2} \leq C'\left(\int_{n}^{n+1}\sigma_{ij}^2(s)ds\right)^{p/2}+C'\left(\int_{n+1}^{n+2}\sigma_{ij}^2(s)ds\right)^{p/2}.
    \]
    Both terms on the right are summable by virtue of Theorem \ref{thm. L^p theorem continuous SVE}. For the case when $p\geq 2$ we invoke Theorem \ref{thm. L^p theorem continuous SVE} and apply Lemma \ref{lem. integrability condition implies summability condition}. Thus,
    \begin{align*}
         \infty > \sum_{n=0}^\infty e^{-p(a_{n+1}-a_n)}\mathbb{E}\left[\sup_{a_n \leq t \leq a_{n+1}}\left|M_i(t)\right|^p\right] > e^{-p}\sum_{n=0}^\infty \mathbb{E}\left[\sup_{a_n \leq t \leq a_{n+1}}\left|M_i(t)\right|^p\right].
    \end{align*}
    Because the right hand side is finite, we can apply the Kolmogorov two series test, yielding
    \[
    \sum_{n=0}^\infty \sup_{a_n \leq t \leq a_{n+1}}\left|M_i(t)\right|^p < \infty.
    \]
    Thus, necessarily
    \[
    \sup_{a_n \leq t \leq a_{n+1}}\left|M_i(t)\right|^p \longrightarrow 0 \text{ as } n \to \infty \quad \text{a.s}.
    \]
    This clearly yields $M_i(t) \to 0 $ almost surely for each $i$. As in the proof of Lemma \ref{lem. L^p of SVE and OU are equivalent} we introduce the process $Z(t)=X(t)-M(t)$, which yields
    \[
    \dot{Z}(t)=(\nu \ast Z )(t) + f(t)+(\nu \ast M)(t)+M(t).
    \]
    By variation of constants
    \[
    Z(t)-(r\ast f)(t)=r(t)\xi +(r\ast \nu \ast M)(t)+(r \ast M)(t).
    \]
    As $r \in L^1(\mathbb{R}_+;\mathbb{R}^{d\times d})$, $\nu \in M(\mathbb{R}_+;\mathbb{R}^{d\times d} )$ and $M \in BC_0(\mathbb{R}_+;\mathbb{R}^d)$, the right hand side will converge to the zero vector upon sending $t \to \infty$. Thus,
    \[
    \lim_{t\to \infty} \left\|X(t)-(r\ast f)(t)\right\| \leq \lim_{t\to \infty} \left(\left\|Z(t)-(r\ast f)(t)\right\| + \|M(t)\|\right)=0 \quad \text{a.s}.
    \]
    
\end{proof}

\section{Stochastic Functional Differential Equations} \label{sec. SFDEs}
Once again, consider a complete filtered probability space $(\Omega,\mathcal{F},\mathcal{F}_t,\mathbb{P})$ supporting an $m$-dimensional standard Brownian motion, $(B_t)_{t\geq0}$. We now study the following $\mathbb{R}^d$-valued stochastic functional differential equation (for some fixed delay parameter $\tau >0$) given by
\begin{align} \label{eq. SFDE}
    & dX(t,\psi)  = \left(f(t)+\int_{[-\tau,0]}\mu(ds)X(t+s,\psi)\right)dt+\sigma(t)dB(t), & \quad t \geq 0;\nonumber \\
     & X(t,\psi)  = \psi(t)  & t \in[-\tau,0].
\end{align}
Here $\psi$ is a $C([-\tau,0];\mathbb{R}^d)$-valued $\mathcal{F}_0$-measurable random variable, $f \in C(\mathbb{R}_+;\mathbb{R}^d)$, $\mu \in M([-\tau,0];\mathbb{R}^{d\times d})$ and $\sigma \in C(\mathbb{R}_+;\mathbb{R}^{d\times m})$. We make the standing assumption
\begin{equation} \label{eq. psi has finite sup norm}
\mathbb{E}\left[\sup_{t \in [-\tau,0]}\|\psi(t)\|^2\right]<\infty,
\end{equation}
to ensure there exists a unique solution of \eqref{eq. SFDE}, see \cite[Chapter 5]{Mao:2008(Book)} for details.
As in the previous section, $(B_t)_{t\geq0}$ is a standard  $m$-dimensional Brownian motion where $ \langle B(t),\textbf{e}_i \rangle$ and $ \langle B(t),\textbf{e}_j \rangle$ are independent for all $i\neq j$. All results from Section \ref{sec. Continuous results} extend to the finite memory problem with only minor modifications, and in fact lead to significant improvements. In the context of functional equations, it is possible to remove any auxiliary assumptions on the resolvent, leading to substantially stronger results. We denote by $r_{\tau}$ the differential resolvent of the measure $\mu$ which solves the matrix equation

\begin{equation} \label{eq. functional resolvent}
  r'_{\tau}(t)=\int_{[-\tau,0]}\mu(ds)r_{\tau}(t+s), \quad t>0; \quad r_{\tau}(0)=I_{d\times d}. 
\end{equation}
Unlike in the case of Volterra equations, we do not need to impose the standing assumption that $r_\tau \in L^1(\mathbb{R}_+;\mathbb{R}^{d\times d})$. We note, however, that this is equivalent to $r_\tau$ obeying the estimate $\|r_\tau(t)\| \leq Ce^{-\alpha t}$, for some $C,\alpha>0$ and an arbitrary norm on $\mathbb{R}^{d \times d}$. We state the following lemma whose proof is identical to that of Lemma \ref{lem. L^p of SVE and OU are equivalent}.
\begin{lemma} \label{lem. L^p SFDE and OU are equivalent}
Let $p \in [1,\infty)$, $X$ be the solution of  \eqref{eq. SFDE}, and $Y$ be the solution of \eqref{eq. continuous OU process}. Assume $r_\tau \in L^1(\mathbb{R}_+;\mathbb{R}^{d\times d})$ and $\psi \in L^p(\Omega;C([-\tau,0];\mathbb{R}^d))$. Then the following statements are true:
\begin{itemize}
    \item[(i)] $\mathbb{E}\|X(\cdot,\psi)\|^p \in L^1([-\tau,\infty);\mathbb{R}) \iff \mathbb{E}\|Y(\cdot)\|^p \in L^1(\mathbb{R}_+;\mathbb{R})$,
    \item[(ii)] $X(\cdot,\psi) \in L^p([-\tau,\infty);\mathbb{R}^d)$ almost surely $\iff$ $Y(\cdot) \in L^p(\mathbb{R}_+;\mathbb{R}^d)$ almost surely.
\end{itemize}
\end{lemma}
\begin{remark}
    Lemma \ref{lem. L^p SFDE and OU are equivalent} provides further evidence to the conjecture that memory has no effect on the integrability of trajectories. In particular, equations with vastly different memory structures (namely, \eqref{eq. continuous SVE} and \eqref{eq. SFDE}) are equivalent at the level of $p$-integrable trajectories.
\end{remark}

\begin{theorem} \label{thm. L^p theorem for continuous SFDE}
    Let $p \in [1,\infty)$, $X$ be the solution of \eqref{eq. SFDE}, and $\psi \in L^p(\Omega;C([-\tau,0];\mathbb{R}^d))$. Then the following statements are true:
    \begin{enumerate}
        \item[(\textbf{A})] If $p \in [2,\infty)$, the following are equivalent:
    \begin{itemize}
        \item[(i)] For each initial function $\psi$, $\mathbb{E}\|X(\cdot,\psi)\|^p \in L^1([-\tau,\infty);\mathbb{R})$,
        \item[(ii)] For each initial function $\psi$, $X(\cdot,\psi) \in L^p([-\tau,\infty);\mathbb{R}^d)$ almost surely,
        \item[(iii)]$r_\tau \in L^1(\mathbb{R}_+;\mathbb{R}^{d\times d})$, $f$ obeys \eqref{cond. f} and $\sigma$ obeys \eqref{cond. sigma p geq 2}.
    \end{itemize}

    \item[(\textbf{B})] If $p \in [1,2)$, the following are equivalent:
    \begin{itemize}
        \item[(i)] For each initial function $\psi$,  $\mathbb{E}\|X(\cdot,\psi)\|^p \in L^1([-\tau,\infty);\mathbb{R})$,
        \item[(ii)]For each initial function $\psi$, $X(\cdot,\psi) \in L^p([-\tau,\infty);\mathbb{R}^d)$ almost surely,
        \item[(iii)]$r_\tau \in L^1(\mathbb{R}_+;\mathbb{R}^{d\times d})$, $f$ obeys \eqref{cond. f} and $\sigma$ obeys \eqref{cond. sigma p<2}.
    \end{itemize}
    \end{enumerate}
\end{theorem}
\begin{remark}
    The reader should note how much stronger Theorem \ref{thm. L^p theorem for continuous SFDE} is compared to Theorems \ref{thm. L^p theorem continuous SVE} and \ref{thm. SVE r L1 converse} for the Volterra equation. For the SFDE, we obtain a significantly cleaner characterisation and a substantially stronger converse regarding the necessity of the condition $r_\tau \in L^1(\mathbb{R}_+;\mathbb{R}^{d\times d})$. A similar phenomenon occurs when analysing the asymptotic behaviour of the paths; see Theorem \ref{thm. SFDE almost sure convergence} below.
\end{remark}

Before proving Theorem \ref{thm. L^p theorem for continuous SFDE} we recall some common notation in the field of functional differential equations. Setting $\sigma=0$ in \eqref{eq. SFDE} yields
\begin{align} \label{eq. Functional x }
\dot{x}(t) =\int_{[-\tau,0]}\mu(du)x(t+u)+f(t), \quad t\geq 0; \quad x(t)=\psi(t), \quad t\leq0,
\end{align}
whose solution is given by
\begin{equation} \label{eq. VOC functional x}
    x(t,\psi)=r_{\tau}(t)\psi(0) + \int_{[-\tau,0]}\mu(ds)\left(\int_s^0r_{\tau}(t+s-u)\psi(u)du\right)+\int_0^t r_{\tau}(t-s)f(s)ds. 
\end{equation}
Let
\begin{equation} \label{eq. x0 functional equation}
    x_0(t,\psi) \coloneqq r_{\tau}(t)\psi(0) + \int_{[-\tau,0]}\mu(ds)\left(\int_s^0r_{\tau}(t+s-u)\psi(u)du\right),
\end{equation}
so we can rewrite \eqref{eq. VOC functional x} as $ x(t,\psi)=x_0(t,\psi)+(r_\tau \ast f)(t)$. Thus, $x_0$ solves equation \eqref{eq. Functional x } with the perturbation term switched off. The so-called characteristic matrix is given by $\Delta(\lambda)=\lambda I_{d \times d} -  \int_{[-\tau,0]} \mu(ds)e^{\lambda s}$ wherein the asymptotic behaviour of $r_\tau$ is governed by the roots of the characteristic equation $\text{det} \Delta(\lambda)=0$. Let $\Lambda = \{ \lambda \in \mathbb{C} : \text{det} \Delta(\lambda)=0 \}$ and define $v_0(\mu) \coloneqq \sup \{ \text{Re}(\lambda): \lambda \in \Lambda\}$ with $\Lambda' = \{ \lambda \in \Lambda : \text{Re}(\lambda) = v_0(\mu) \}$.
 \begin{proof}[Proof of Theorem \ref{thm. L^p theorem for continuous SFDE}]
     An argument identical to that in the proof of Theorem \ref{thm. L^p theorem continuous SVE} gives $(iii) \implies (i)$ for both (\textbf{A}) and (\textbf{B}). In order to show $(i) \implies (ii)$ we need to show that $(i) \implies r_\tau \in L^1(\mathbb{R}_+;\mathbb{R}^{d\times d})$ to which we could then follow the line of proof once again from Theorem \ref{thm. L^p theorem continuous SVE}. Thus we assume $(i)$ and aim to show $r_\tau \in L^1(\mathbb{R}_+;\mathbb{R}^{d\times d})$. First we note that for an arbitrary deterministic $\psi$ we have $X(t,\psi)-X(t,0)=x_0(t,\psi)$ and hence
     \[
     \int_0^\infty\mathbb{E}\|x_0(t,\psi)\|^pdt \leq C\int_0^\infty\mathbb{E}\|X(t,0)\|^pdt+ C\int_0^\infty\mathbb{E}\|X(t,\psi)\|^pdt<\infty,
     \]
     for some constant $C>0$. As $\psi$ is deterministic we must have $x_0(\cdot,\psi) \in L^p([-\tau,\infty);\mathbb{R}^d)$. Let $\lambda \in \Lambda'$, then by exercise I.3.8 in \cite{Diekmann} we have for some $v \in \mathbb{R}^d\backslash \{0\}$, $x_0(t,\psi)=\exp({\text{Re}(\lambda)t})  v$. If $v_0(\mu)\geq 0$ this contradicts $x_0(\cdot,\psi) \in L^p([-\tau,\infty);\mathbb{R}^d)$ and so we must have $v_0(\mu)< 0$, but this yields the exponential estimate $\|r_\tau(t)\| \leq Ce^{-\alpha t}$, for some $C,\alpha>0$ and so $r_\tau \in L^1(\mathbb{R}_+;\mathbb{R}^{d\times d})$ as required. Thus, we follow the proof of Theorem \ref{thm. L^p theorem continuous SVE} verbatim to claim $(i) \implies (ii)$ for both (\textbf{A}) and (\textbf{B}). In the same manner one can show $(ii) \implies r_\tau \in L^1(\mathbb{R}_+;\mathbb{R}^{d\times d})$ and once again follow the proof of Theorem \ref{thm. L^p theorem continuous SVE} to claim $(ii) \implies (iii)$ for both (\textbf{A}) and (\textbf{B}).
 \end{proof}
 Our final theorem regarding SFDEs is concerned with almost sure convergence to zero in the regime where $\sigma$ is diagonal and once again, we improve significantly on the Volterra case by removing all assumptions on the resolvent. We note that a version of Theorem \ref{thm. pathwise behaviour of X} also holds for the functional equation but we do not state this as a theorem.

 \begin{theorem} \label{thm. SFDE almost sure convergence}
Let $\psi$ obey \eqref{eq. psi has finite sup norm}, $X$ be the solution of  \eqref{eq. continuous SVE}, $m=d$, and $\sigma$ be a diagonal matrix. Then the following are equivalent:
\begin{itemize}

    \item[(i)] $r_\tau \in L^1(\mathbb{R}_+;\mathbb{R}^{d\times d})$, each component of $f$ and $\sigma$ satisfy \eqref{eq. conditions on f sigma for asymptotic convergence},
    \item[(ii)] For each initial function $\psi$, $\|X(t,\psi)\| \longrightarrow 0 $ as $t \to \infty$, almost surely.

\end{itemize}
\end{theorem}
\begin{proof}[Proof of Theorem \ref{thm. SFDE almost sure convergence}]
The proof of $(i) \implies (ii)$ follows directly from Theorem \ref{thm. characterising convergence of X a.s. diagonal sigma} with only trivial modifications.  For the reverse implication, first one needs to show $(ii) \implies r_\tau \in L^1(\mathbb{R}_+;\mathbb{R}^{d\times d})$, which follows from identical arguments as in the proof of Theorem \ref{thm. L^p theorem for continuous SFDE}. Thus we can once again (with only trivial modifications) follow the proof of Theorem \ref{thm. characterising convergence of X a.s. diagonal sigma} to yield the claim. 
\end{proof}

\section{Examples} \label{sec. Examples}
This section is devoted to providing explicit examples of ill-behaved perturbation functions which highlight the utility of the theory presented. This dispels any thought that the conditions \eqref{cond. f} and \eqref{cond. sigma p geq 2} are superfluous and in fact just equivalent to imposing that $f,\sigma^2 \in L^p$ in the appropriate sense. Indeed, this is not the case. All functions presented in this section will be scalar-valued. In \cite{AL:2023(AppliedMathLetters)} the authors provided the example
\[
f(t)=e^{\alpha t}\sin(e^{\beta t}),
\]
where $0< \alpha < \beta$ and show $\int_\cdot^{\cdot+\theta}f(s)ds \in L^p(\mathbb{R}_+;\mathbb{R})$ despite the fact that $\int_0^\infty|f(s)|^pds=\infty$ for any $p\geq1.$ This takes care of the deterministic perturbation, however, for the stochastic perturbation we are concerned with $\sigma^2$ rather than $\sigma$ itself. Thus, positivity means such oscillatory functions as above cannot be considered. Instead, one must think about functions that have very large deviations over very small intervals, we construct such a function below.

For each $n\in \mathbb{N}$, let $a_n<1/2$ and $h_n$ be positive sequences and suppose that $g(t)=0$ for $[n,n+a_n]$ and $[n+1-a_n, n+1]$ and on $[n+a_n,n+1/2]$, $g$ is linear with $g(n+a_n)=0$ 
and $g(n+1/2)=h_n$, while on $[n+1/2,n+1-a_n]$, $g$ is linear with $g(n+1-a_n)=0$. Then $g$ is continuous, has a spike of width $1-2a_n$ and maximal height $h_n$. An explicit example is
\[
g(t) \coloneqq 
\begin{cases}
    0,& t \in [0,2],\\
    0,& t \in [n,n+a_n],\\
    \frac{h_n}{(\frac{1}{2}-a_n)} t-\frac{h_n(n+a_n)}{(\frac{1}{2}-a_n)},& t \in [n+a_n,n+\frac{1}{2}],\\
    \frac{-h_n}{(\frac{1}{2}-a_n)} t+\frac{h_n(n+1-a_n)}{(\frac{1}{2}-a_n)},& t \in [n+\frac{1}{2},n+1-a_n],\\
    0,& t \in [n+1-a_n,n+1], \\
\end{cases}
\]
$n\geq2,$ with $h_n \coloneqq n^\beta$ and $a_n \coloneqq \frac{1}{2}-\frac{1}{n^{\beta +1}}$, where $\beta>0$.

The following proposition shows this function has the desired behaviour.
\begin{prop} \label{prop. g spikey example}
    Let $g: \mathbb{R}_+\to \mathbb{R}_+$ be defined as above, then for each $p \in (1,\infty)$ we have:
    \begin{align*}
        \int_0^\infty\left|\int_t^{t+1}g(s)ds\right|^pdt < \infty; \qquad  \int_0^\infty |g(s)|^pds = \infty.
    \end{align*}
\end{prop}
\begin{proof}[Proof]
    We calculate the integral of $g$ over the interval $[n,n+1]$ for some fixed $n\geq2$. By definition
    \begin{align*}
       \int_{n}^{n+1}g(s)ds& = \int_{n+a_n}^{n+\frac{1}{2}}\frac{h_n}{(\frac{1}{2}-a_n)} s-\frac{h_n(n+a_n)}{(\frac{1}{2}-a_n)}ds+\int_{n+\frac{1}{2}}^{n+1-a_n}\frac{-h_n}{(\frac{1}{2}-a_n)} s+\frac{h_n(n+1-a_n)}{(\frac{1}{2}-a_n)}ds\\
       & = \frac{h_n(n+\frac{1}{2})^2}{2(\frac{1}{2}-a_n)} -\frac{h_n(n+a_n)(n+\frac{1}{2})}{(\frac{1}{2}-a_n)}+\frac{h_n(n+a_n)^2}{2(\frac{1}{2}-a_n)} \\
       & \quad +\frac{h_n(n+1-a_n)^2}{2(\frac{1}{2}-a_n)} +\frac{h_n(n+\frac{1}{2})^2}{2(\frac{1}{2}-a_n)} -\frac{h_n(n+1-a_n)(n+\frac{1}{2})}{(\frac{1}{2}-a_n)}\\
       & = \frac{h_n}{2(\frac{1}{2}-a_n)}\left((n+\frac{1}{2}-n-a_n)^2+(n+1-a_n-n-\frac{1}{2})^2\right)\\
       & = h_n(\frac{1}{2}-a_n).
    \end{align*}
    Alternatively, the integral over $[n,n+1]$ is exactly the area of a triangle of height $h_n$ and base $1-2a_n$. Hence the familiar formula for the area of a triangle yields $h_n(\frac{1}{2}-a_n)$, confirming the above calculations. Recall, $h_n=n^{\beta}$ and $a_n=\frac{1}{2}-\frac{1}{n^{\beta+1}}$. Thus,
    \begin{align*}
        \int_{n}^{n+1}g(s)ds= h_n(\frac{1}{2}-a_n)=\frac{n^{\beta}}{n^{\beta+1}}=\frac{1}{n}.
    \end{align*}
    Next, consider
    \begin{align*}
        \int_0^\infty\left|\int_t^{t+1}g(s)ds\right|^pdt & = \sum_{n=2}^\infty\int_{n}^{n+1}\left(\int_{t}^{t+1}g(s)ds\right)^pdt\\
        & \leq \sum_{n=2}^\infty\int_{n}^{n+1}\left(\int_{n}^{n+2}g(s)ds\right)^pdt\\
        & = \sum_{n=2}^\infty \left(\int_{n}^{n+1}g(s)ds+\int_{n+1}^{n+2}g(s)ds\right)^p\\
        & \leq C_p \left(\sum_{n=2}^{\infty}\frac{1}{n^p}+\frac{1}{(n+1)^p}\right) < \infty,
    \end{align*}
where $C_p>0$ is a constant depending on $p$. Hence the first assertion is proven. For the second assertion consider the estimation
\begin{align*}
    \int_0^{\infty}g(s)^pds \geq \sum_{n=1}^{\infty}\int_{n}^{n+1}g(s)^p \chi_{\{g(s) \geq 1\}}(s)ds.
\end{align*}
Thus, we need only show the series on the right is divergent and then we are done. As $p >1$ we have
\[
\int_{n}^{n+1}g(s)^p \chi_{\{g(s) \geq 1\}}(s)ds \geq \int_{n}^{n+1}g(s) \chi_{\{g(s) \geq 1\}}(s)ds.
\]
As $\beta >0 $ we have $\limsup_{t\to \infty}g(t)=\infty$, moreover the running maximum is non-decreasing and hence there exists an $m \in \mathbb{N}$ such that $\sup_{t \in [n,n+1]}g(t)>1$ for all $n\geq m$. Fix such an $m$ (recall $g$ is continuous, so the supremum is always attained). One need not evaluate the above integral directly but rather observe it is exactly $\frac{1}{n}$ minus the area of two identical triangles of height $1$ and base $\frac{1}{n^{2\beta+1}}$ (because the slope of $g(t)$ for $t \in [n+a_n,n+\frac{1}{2}]$ is $n^{2\beta+1}$). Hence for $n\geq m$
\[
\int_{n}^{n+1}g(s) \chi_{\{g(s) \geq 1\}}(s)ds= \frac{1}{n}-\frac{1}{n^{2\beta+1}}.
\]
Because $\beta>0$, the above is not a summable sequence. Thus,
\[
\int_0^{\infty}g(s)^pds\geq \sum_{n=m}^{\infty}\int_{n}^{n+1}g(s)^p \chi_{\{g(s) \geq 1\}}(s)ds \geq \sum_{n=m}^{\infty}\left(\frac{1}{n}-\frac{1}{n^{2\beta+1}}\right) = \infty,
\]
as required.
\end{proof}
Consider the following scalar equation
\[
dX(t)=\left(e^{\alpha t}\sin(e^{\beta t})+\int_{[0,t]}\nu(ds)X(t-s)\right)dt+\sigma(t)dB(t),
\]
with $\sigma(t) \coloneqq \sqrt{g(t)}$ where $g(t)$ is defined as in Proposition \ref{prop. g spikey example}. Theorem \ref{thm. L^p theorem continuous SVE} implies $X \in L^{p}(\mathbb{R}_+;\mathbb{R})$ almost surely despite the fact that $\int_{0}^\infty|f(s)|^pds=\infty$ and $\int_0^\infty |\sigma^2(s)|^{p/2}ds=\infty$ for any $p\geq 2$.


\section{Conclusion} \label{sec. Conclusions}
In this article, we provided characterisations of when perturbed integro-differential Volterra equations with additive noise admit almost surely $p$-integrable trajectories, when the associated $p$-th mean process belongs to $L^1(\mathbb{R}_+;\mathbb{R}_+)$, and when solutions to the analogous discrete-time equations are almost surely $p$-summable.

The main driver of these results is the connection between the Volterra equation \eqref{eq. continuous SVE} and the associated SDE \eqref{eq. continuous OU process}. The results of this paper suggest that the qualitative (and possibly quantitative) behaviour of such Volterra equations is equivalent to that of a considerably simpler SDE, and we conjecture that this equivalence may extend to the identification of the top Lyapunov exponent.

In addition to integrability properties, we also established almost sure convergence to zero in various regimes and provided a characterisation of such convergence in a special case. As outlined in Section \ref{sec. SFDEs}, the relationship between Volterra equations and SFDEs is strong. Furthermore, when passing to the finite memory setting, one obtains substantially stronger results than in the Volterra case.

\begin{appendices}

\section{Appendix}

\begin{proof}[Proof of Proposition \ref{prop. examples of random variables}]
    Suppose the distribution of $\xi$ has two isolated atoms at $\alpha$ and $\beta$; then we can find epsilon balls such that $\mathbb{E}\xi \chi_{\{\xi \in B_\varepsilon(\alpha)\}}=\alpha \mathbb{P}(\xi =\alpha)$ and $\mathbb{E}\xi \chi_{\{\xi \in B_\varepsilon(\beta)\}}=\beta \mathbb{P}(\xi =\beta)$. Assume
    \[
    \mathbb{P}(\xi =\alpha) \mathbb{E}\xi \chi_{\{\xi \in B_\varepsilon(\beta)\}}= \mathbb{P}(\xi =\beta)  \mathbb{E}\xi \chi_{\{\xi \in B_\varepsilon(\alpha)\}}.
    \]
    Clearly this forces $\beta=\alpha$ which is a contradiction. Thus, we take the two bounded Borel sets to be the epsilon balls defined above and we have $\xi \in \mathbb{D}$. Now assume the distribution of $\xi$ has an absolutely continuous part, and denote its density by $f$. Fix an arbitrary non-zero $t$ in the interior of the support of $f$ and choose $a<b$ such that $\mathbb{P}(\xi \in (a,t))=\mathbb{P}(\xi \in (t,b))$\footnote{This can always be done due to the continuity of the distribution function restricted to the support of $f$.}. Assume
    \[
    \mathbb{P}(\xi \in (a,t))   \mathbb{E}\xi \chi_{\{\xi \in (t,b)\}}= \mathbb{P}(\xi \in (t,b))  \mathbb{E}\xi \chi_{\{\xi \in (a,t)\}}, 
    \]
    which yields
    \[
    \int_t^b x  f(x)dx=\int_a^tx  f(x)dx.
    \]
    A simple estimate yields
    \[
    t\int_t^bf(x)dx\leq \int_a^txf(x)dx.
    \]
    But by construction, we must have
    \[
    t\int_a^tf(x)dx\leq \int_a^txf(x)dx.
    \]
    This yields the inequality
    \[
    0\leq \int_a^t(t-x)f(x)dx\leq 0.
    \]
    Thus, the integrand must be zero. But this is impossible, hence we have the desired contradiction. Thus, we choose the Borel sets to be $(a,t)$ and $(t,b)$ which gives $\xi \in \mathbb{D}$ as required.
    
\end{proof}
\begin{proof}[Proof of Lemma \ref{lem. integrability condition implies summability condition} ]
Let $\theta=2\beta$. For all $t \in [a_n,a_{n+1}]$ we have $[a_{n+1},a_n+2\beta] \subset [t,t+2\beta]$. Thus,
\begin{align*}
    \infty > \int_0^\infty \left(\int_t^{t+2\beta}f(s)ds\right)^p dt & = \sum_{n=0}^\infty \int_{a_n}^{a_{n+1}}\left(\int_t^{t+2\beta}f(s)ds\right)^p dt\\
    & \geq \sum_{n=0}^\infty \int_{a_n}^{a_{n+1}}\left(\int_{a_{n+1}}^{a_{n}+2\beta}f(s)ds\right)^p dt\\
    & =\sum_{n=0}^\infty (a_{n+1}-a_n)\left(\int_{a_{n+1}}^{a_{n}+2\beta}f(s)ds\right)^p\\
    & \geq \alpha \sum_{n=0}^\infty \left(\int_{a_{n+1}}^{a_{n}+2\beta}f(s)ds\right)^p\\
    & \geq \alpha \sum_{n=0}^\infty \left(\int_{a_{n+1}}^{a_{n+2}}f(s)ds\right)^p,
\end{align*}
where the last inequality comes from the fact $a_{n+2}\leq \beta+a_{n+1}\leq 2\beta + a_{n}$.
    
\end{proof}
\end{appendices}
\bibliographystyle{unsrt}
\bibliography{SS_Characterisation_SVE_additive_noise_Lp}

@article {AP:2002(ECP),
    AUTHOR = {Appleby, J. A. D.},
     TITLE = {Almost sure stability of linear {I}t\^{o}-{V}olterra equations
              with damped stochastic perturbations},
   JOURNAL = {Electron. Comm. Probab.},
  FJOURNAL = {Electronic Communications in Probability},
    VOLUME = {7},
      YEAR = {2002},
     PAGES = {223--234},
      ISSN = {1083-589X},
   MRCLASS = {60F15 (60H20)},
  MRNUMBER = {1952184},
MRREVIEWER = {Xue\ Rong\ Mao},
       DOI = {10.1214/ECP.v7-1063},
       URL = {https://doi.org/10.1214/ECP.v7-1063},
}

@article {AP:2004(SubExpItoVol),
    AUTHOR = {Appleby, J. A. D.},
     TITLE = {Subexponential solutions of linear {I}t\^{o}-{V}olterra
              equations with a damped perturbation},
   JOURNAL = {Funct. Differ. Equ.},
  FJOURNAL = {Functional Differential Equations},
    VOLUME = {11},
      YEAR = {2004},
    NUMBER = {1-2},
     PAGES = {5--10},
      ISSN = {0793-1786},
   MRCLASS = {34K20 (34K50 60H10)},
  MRNUMBER = {2054754},
}

@article {AP:2021,
    AUTHOR = {Appleby, J. A. D.},
     TITLE = {Mean square characterisation of a stochastic {V}olterra
              integrodifferential equation with delay},
   JOURNAL = {Int. J. Dyn. Syst. Differ. Equ.},
  FJOURNAL = {International Journal of Dynamical Systems and Differential
              Equations},
    VOLUME = {11},
      YEAR = {2021},
    NUMBER = {3-4},
     PAGES = {194--226},
      ISSN = {1752-3583,1752-3591},
   MRCLASS = {60H20 (34K50 45J05)},
  MRNUMBER = {4318163},
       DOI = {10.1504/IJDSDE.2021.117374},
       URL = {https://doi.org/10.1504/IJDSDE.2021.117374},
}

@article {ACR:2011(DCDS),
    AUTHOR = {Appleby, J. A. D. and Cheng, J. and Rodkina, A.},
     TITLE = {Characterisation of the asymptotic behaviour of scalar linear
              differential equations with respect to a fading stochastic
              perturbation},
   JOURNAL = {Discrete Contin. Dyn. Syst.},
  FJOURNAL = {Discrete and Continuous Dynamical Systems. Series A},
      YEAR = {2011},
     PAGES = {79--90},
      ISSN = {1078-0947,1553-5231},
      ISBN = {978-1-60133-007-9; 1-60133-007-3},
   MRCLASS = {60H10 (93D09 93E15)},
  MRNUMBER = {3012136},
MRREVIEWER = {G.\ N.\ Mil\cprime shte\u{\i}n},
}

@article {AF:2003(EJP),
    AUTHOR = {Appleby, J. A. D. and Freeman, A.},
     TITLE = {Exponential asymptotic stability of linear
              {I}t\^{o}-{V}olterra equations with damped stochastic
              perturbations},
   JOURNAL = {Electron. J. Probab.},
  FJOURNAL = {Electronic Journal of Probability},
    VOLUME = {8},
      YEAR = {2003},
     PAGES = {no. 22, 22},
      ISSN = {1083-6489},
   MRCLASS = {60H10 (34K20 34K50 60H20)},
  MRNUMBER = {2041823},
MRREVIEWER = {Marcus\ R. W. Martin},
       DOI = {10.1214/EJP.v8-179},
       URL = {https://doi.org/10.1214/EJP.v8-179},
}

@misc{AL:2024(Pantograph),
      title={Characterisation of asymptotic behaviour of perturbed deterministic and stochastic pantograph equations}, 
      author={John A. D. Appleby and E. Lawless},
      year={2025},
      eprint={2410.16435},
      archivePrefix={arXiv},
      primaryClass={math.CA},
      url={https://arxiv.org/abs/2410.16435}, 
}

@article {AL:2023(AppliedNumMath),
    AUTHOR = {Appleby, J. A. D. and Lawless, E.},
     TITLE = {Mean square asymptotic stability characterisation of perturbed
              linear stochastic functional differential equations},
   JOURNAL = {Appl. Numer. Math.},
  FJOURNAL = {Applied Numerical Mathematics. An IMACS Journal},
    VOLUME = {200},
      YEAR = {2024},
     PAGES = {80--109},
      ISSN = {0168-9274,1873-5460},
   MRCLASS = {93E15 (45M10 45R05 60H10)},
  MRNUMBER = {4731096},
       DOI = {10.1016/j.apnum.2023.07.005},
       URL = {https://doi.org/10.1016/j.apnum.2023.07.005},
}

@article{AL:2023(AppliedMathLetters),
  author = {J. A. D. Appleby and E. Lawless},
  title = {Solution space characterisation of perturbed linear Volterra integrodifferential convolution equations: The ${L}^p$ case},
    journal = {Applied Mathematics Letters},
    volume = {146},
    pages = {108825},
    issn = {0893-9659},
    year = {2023},
    doi = {https://doi.org/10.1016/j.aml.2023.108825}
}

@incollection {AL:2023(ICDEA),
    AUTHOR = {Appleby, J. A. D. and Lawless, E.},
     TITLE = {On the dynamics and asymptotic behaviour of the mean square of
              scalar linear stochastic difference equations},
 BOOKTITLE = {Advances in discrete dynamical systems, difference equations
              and applications},
    SERIES = {Springer Proc. Math. Stat.},
    VOLUME = {416},
     PAGES = {25--60},
 PUBLISHER = {Springer, Cham},
      YEAR = {2023},
      ISBN = {978-3-031-25224-2; 978-3-031-25225-9},
   MRCLASS = {39A50 (39A06)},
  MRNUMBER = {4606673},
       DOI = {10.1007/978-3-031-25225-9\_2},
       URL = {https://doi.org/10.1007/978-3-031-25225-9_2},
}

@article {ApRie:2006(SAA),
    AUTHOR = {Appleby, J. A. D. and Riedle, M.},
     TITLE = {Almost sure asymptotic stability of stochastic {V}olterra
              integro-differential equations with fading perturbations},
   JOURNAL = {Stoch. Anal. Appl.},
  FJOURNAL = {Stochastic Analysis and Applications},
    VOLUME = {24},
      YEAR = {2006},
    NUMBER = {4},
     PAGES = {813--826},
      ISSN = {0736-2994,1532-9356},
   MRCLASS = {60H10 (34K20 34K50 60H20)},
  MRNUMBER = {2241094},
MRREVIEWER = {Mark\ A.\ McKibben},
       DOI = {10.1080/07362990600753536},
       URL = {https://doi.org/10.1080/07362990600753536},
}

@article{AP:2017,
  author = {J. A. D. Appleby and D. D. Patterson},
  title = {Large fluctuations and growth rates of linear Volterra summation equations},
  journal = {Differential Equations and Applications},
  volume = {23},
  number = {6},
  pages = {1047–1080},
  year = {2017},
}

@article{Ap-Rie-Rod:2009,
  author = {J. A. D. Appleby and M. Riedle and A. Rodkina},
  title = {On Asymptotic Stability of linear stochastic
  Volterra difference equations with respect to a fading perturbation},
  journal = {Advanced Studies in Pure Mathematics},
  publisher = {Mathematical Society of Japan},
  volume = {53},
  pages = {271-282},
  year = {2009},
}

@article {BerRod:2006(JDEA),
    AUTHOR = {Berkolaiko, G. and Rodkina, A.},
     TITLE = {Almost sure convergence of solutions to nonhomogeneous
              stochastic difference equation},
   JOURNAL = {J. Difference Equ. Appl.},
  FJOURNAL = {Journal of Difference Equations and Applications},
    VOLUME = {12},
      YEAR = {2006},
    NUMBER = {6},
     PAGES = {535--553},
      ISSN = {1023-6198,1563-5120},
   MRCLASS = {39A10 (39A11 60H99)},
  MRNUMBER = {2240374},
MRREVIEWER = {Leonid\ E.\ Shaikhet},
       DOI = {10.1080/10236190600574093},
       URL = {https://doi.org/10.1080/10236190600574093},
}

@article {ChanWilliams:1989,
    AUTHOR = {Chan, T. and Williams, D.},
     TITLE = {An ``excursion'' approach to an annealing problem},
   JOURNAL = {Math. Proc. Cambridge Philos. Soc.},
  FJOURNAL = {Mathematical Proceedings of the Cambridge Philosophical
              Society},
    VOLUME = {105},
      YEAR = {1989},
    NUMBER = {1},
     PAGES = {169--176},
      ISSN = {0305-0041,1469-8064},
   MRCLASS = {60H10 (60F10 60J60 93E15)},
  MRNUMBER = {966154},
MRREVIEWER = {Vigirdas\ Mackevi\v{c}ius},
       DOI = {10.1017/S030500410000150X},
       URL = {https://doi.org/10.1017/S030500410000150X},
}

@book{Cor90b,
  author = {C. Corduneanu},
  title = {Integral equations and applications},
  publisher = {Cambridge Univ. Press, New York},
  year = {1990},
}

@article {Diblik:2011(AAA),
    AUTHOR = {Diblík, J. and Růžičková, M. and Schmeidel, E. and Zbąszyniak, M.},
     TITLE = {Weighted asymptotically periodic solutions of linear
              {V}olterra difference equations},
   JOURNAL = {Abstr. Appl. Anal.},
  FJOURNAL = {Abstract and Applied Analysis},
      YEAR = {2011},
     PAGES = {Art. ID 370982, 14},
      ISSN = {1085-3375,1687-0409},
   MRCLASS = {39A23 (39A30)},
  MRNUMBER = {2795073},
MRREVIEWER = {Rigoberto\ L.\ Medina},
       DOI = {10.1155/2011/370982},
       URL = {https://doi.org/10.1155/2011/370982},
}

@book {Diekmann,
    AUTHOR = {Diekmann, O. and van Gils, S. A. and Verduyn Lunel,
              S. M. and Walther, H.-O.},
     TITLE = {Delay equations},
    SERIES = {Applied Mathematical Sciences},
    VOLUME = {110},
      NOTE = {Functional, complex, and nonlinear analysis},
 PUBLISHER = {Springer-Verlag, New York},
      YEAR = {1995},
     PAGES = {xii+534},
      ISBN = {0-387-94416-8},
   MRCLASS = {34-02 (34K05 34K15 34K20)},
  MRNUMBER = {1345150},
MRREVIEWER = {J.\ M.\ Cushing},
       DOI = {10.1007/978-1-4612-4206-2},
       URL = {https://doi.org/10.1007/978-1-4612-4206-2},
}

@article {GyoriReynolds:2010(periodicSolutions),
    AUTHOR = {Győri, I. and Reynolds, D. W.},
     TITLE = {On asymptotically periodic solutions of linear discrete
              {V}olterra equations},
   JOURNAL = {Fasc. Math.},
  FJOURNAL = {Polytechnica Posnaniensis. Institutum Mathematicum. Fasciculi
              Mathematici},
    NUMBER = {44},
      YEAR = {2010},
     PAGES = {53--67},
      ISSN = {0044-4413},
   MRCLASS = {39A06 (39A22 39A23)},
  MRNUMBER = {2722631},
MRREVIEWER = {Claudio\ Cuevas},
}

@book{GLS,
  author = {G. Gripenberg and S.-O. Londen and O. Staffans},
  title = {Volterra Integral and Functional Equations},
    series = {Encyclopedia of Mathematics and it's Applications},
  publisher = {Cambridge University Press},
  year = {1990},
  }

@book {KolMysh:1999(FDE),
    AUTHOR = {Kolmanovskii, V. B. and Myshkis, A.},
     TITLE = {Introduction to the theory and applications of
              functional-differential equations},
    SERIES = {Mathematics and its Applications},
    VOLUME = {463},
 PUBLISHER = {Kluwer Academic Publishers, Dordrecht},
      YEAR = {1999},
     PAGES = {xvi+648},
      ISBN = {0-7923-5504-0},
   MRCLASS = {34Kxx (49J25 93C23)},
  MRNUMBER = {1680144},
MRREVIEWER = {Peter\ Dormayer},
       DOI = {10.1007/978-94-017-1965-0},
       URL = {https://doi.org/10.1007/978-94-017-1965-0},
}

@book {KolNos:1986(SFDE),
    AUTHOR = {Kolmanovskii, V. B. and Nosov, V. R.},
     TITLE = {Stability of functional-differential equations},
    SERIES = {Mathematics in Science and Engineering},
    VOLUME = {180},
 PUBLISHER = {Academic Press, Inc. [Harcourt Brace Jovanovich, Publishers],
              London},
      YEAR = {1986},
     PAGES = {xiv+217},
      ISBN = {0-12-417940-1; 0-12-417941-X},
   MRCLASS = {34-01 (34K20 92A15 93D20)},
  MRNUMBER = {860947},
MRREVIEWER = {R.\ F.\ Datko},
}

@article {Mao:2000(SAA),
    AUTHOR = {Mao, X.},
     TITLE = {Stability of stochastic integro-differential equations},
   JOURNAL = {Stochastic Anal. Appl.},
  FJOURNAL = {Stochastic Analysis and Applications},
    VOLUME = {18},
      YEAR = {2000},
    NUMBER = {6},
     PAGES = {1005--1017},
      ISSN = {0736-2994,1532-9356},
   MRCLASS = {60H10 (45R05 93E15)},
  MRNUMBER = {1794076},
MRREVIEWER = {Vigirdas\ Mackevi\v{c}ius},
       DOI = {10.1080/07362990008809708},
       URL = {https://doi.org/10.1080/07362990008809708},
}

@book {Mao:2008(Book),
    AUTHOR = {Mao, X.},
     TITLE = {Stochastic differential equations and applications},
   EDITION = {Second},
 PUBLISHER = {Horwood Publishing Limited, Chichester},
      YEAR = {2008},
     PAGES = {xviii+422},
      ISBN = {978-1-904275-34-3},
   MRCLASS = {60-02 (34F05 34K50 60H10 60H30 91B28)},
  MRNUMBER = {2380366},
       DOI = {10.1533/9780857099402},
       URL = {https://doi.org/10.1533/9780857099402},
}

@article {MaoRie:2006(SCL),
    AUTHOR = {Mao, X. and Riedle, M.},
     TITLE = {Mean square stability of stochastic {V}olterra
              integro-differential equations},
   JOURNAL = {Systems Control Lett.},
  FJOURNAL = {Systems \& Control Letters},
    VOLUME = {55},
      YEAR = {2006},
    NUMBER = {6},
     PAGES = {459--465},
      ISSN = {0167-6911,1872-7956},
   MRCLASS = {34K20 (34K50 45J05 93D05)},
  MRNUMBER = {2216754},
MRREVIEWER = {Jialin\ Hong},
       DOI = {10.1016/j.sysconle.2005.09.009},
       URL = {https://doi.org/10.1016/j.sysconle.2005.09.009},
}

@book {Raff:2018(QVDE),
    AUTHOR = {Raffoul, Y. N.},
     TITLE = {Qualitative theory of {V}olterra difference equations},
 PUBLISHER = {Springer, Cham},
      YEAR = {2018},
     PAGES = {xiv+324},
      ISBN = {978-3-319-97189-6; 978-3-319-97190-2},
   MRCLASS = {39A60 (39A12 39A22 39A23 39A30 39A70 92D25)},
  MRNUMBER = {3838344},
       DOI = {10.1007/978-3-319-97190-2},
       URL = {https://doi.org/10.1007/978-3-319-97190-2},
}

@book {shaik:2013,
    AUTHOR = {Shaikhet, L.},
     TITLE = {Lyapunov functionals and stability of stochastic functional
              differential equations},
 PUBLISHER = {Springer, Cham},
      YEAR = {2013},
     PAGES = {xii+342},
      ISBN = {978-3-319-00100-5; 978-3-319-00101-2},
   MRCLASS = {34-02 (34K20 34K50 93-02 93D30 93E15)},
  MRNUMBER = {3076210},
MRREVIEWER = {Fuke\ Wu},
       DOI = {10.1007/978-3-319-00101-2},
       URL = {https://doi.org/10.1007/978-3-319-00101-2},
}

@book{shaik:2011,
    AUTHOR = {Shaikhet, L.},
     TITLE = {Lyapunov functionals and stability of stochastic difference
              equations},
 PUBLISHER = {Springer, London},
      YEAR = {2011},
     PAGES = {xii+370},
      ISBN = {978-0-85729-684-9; 978-0-85729-685-6},
   MRCLASS = {39-02 (34D20 34K20 39A30 39A50 93D30)},
  MRNUMBER = {3015017},
MRREVIEWER = {Henri\ Schurz},
       DOI = {10.1007/978-0-85729-685-6},URL = {https://doi.org/10.1007/978-0-85729-685-6},
}

@article {Shai:1997(AML),
    AUTHOR = {Shaikhet, L.},
     TITLE = {Necessary and sufficient conditions of asymptotic mean square
              stability for stochastic linear difference equations},
   JOURNAL = {Appl. Math. Lett.},
  FJOURNAL = {Applied Mathematics Letters. An International Journal of Rapid
              Publication},
    VOLUME = {10},
      YEAR = {1997},
    NUMBER = {3},
     PAGES = {111--115},
      ISSN = {0893-9659,1873-5452},
   MRCLASS = {93D05 (39A99 93E15)},
  MRNUMBER = {1457650},
MRREVIEWER = {Sheldon\ P.\ Gordon},
       DOI = {10.1016/S0893-9659(97)00045-1},
       URL = {https://doi.org/10.1016/S0893-9659(97)00045-1},
}

@article{ShaiRob:2011(DCDIS),
    AUTHOR = {Shaikhet, L. and Roberts, J. A.},
     TITLE = {Asymptotic stability analysis of a stochastic {V}olterra
              integro-differential equation with fading memory},
   JOURNAL = {Dyn. Contin. Discrete Impuls. Syst. Ser. B Appl. Algorithms},
  FJOURNAL = {Dynamics of Continuous, Discrete \& Impulsive Systems. Series
              B. Applications \& Algorithms},
    VOLUME = {18},
      YEAR = {2011},
    NUMBER = {6},
     PAGES = {749--770},
      ISSN = {1492-8760,1918-2538},
   MRCLASS = {34K20 (39A30 39A50 45R05 60H20 93E15)},
  MRNUMBER = {2896563},
MRREVIEWER = {Nguyen Tien Dung},
}

@article{SY67a,
  author = {A. Strauss and J. A. Yorke},
  title = {Perturbation theorems for ordinary differential equations},
  journal = {Journal of Differential Equations},
  volume = {3},
  pages = {15--30},
  year = {1967},
}

@article{SY67b,
  author = {A. Strauss and J. A. Yorke},
  title = {On asymptotically autonomous differential equations},
  journal = {Mathematical Systems Theory},
  volume = {1},
  pages = {175--182},
  year = {1967},
}

\end{document}